\documentclass[12pt,reqno,oneside]{amsart}
\usepackage{mathdots}
\usepackage{amsmath}
\usepackage{amssymb}
\usepackage{amsthm}
\usepackage{graphics}
\usepackage{eucal}
\usepackage{bm}
\usepackage[colorlinks]{hyperref}
\usepackage[capitalize,nameinlink,noabbrev]{cleveref}
\usepackage{mathrsfs}
\usepackage[normalem]{ulem}
\usepackage{color}
\usepackage[dvipsnames]{xcolor}
\usepackage{mathtools}
\usepackage{geometry}

\newtheorem{theorem}{Theorem}[section]
\newtheorem{proposition}{Proposition}[section]
\newtheorem{lemma}{Lemma}[section]
\newtheorem{corollary}{Corollary}[section]

\theoremstyle{remark}
\newtheorem{remark}{Remark}[section]
\theoremstyle{definition}

\numberwithin{equation}{section}

\newcommand{\Spec}{\operatorname{Spec}}
\newcommand{\vv}[1]{{\mathbf{#1}}}

\newcommand{\SL}{\operatorname{SL}}
\newcommand{\diag}{\operatorname{diag}}

\newcommand{\GL}{\operatorname{GL}}

\newcommand{\bx}{\mathbf{x}}
\newcommand{\bU}{\mathbf{U}}
\newcommand{\bT}{\mathbf{T}}

\newcommand{\cL}{\mathcal{L}}
\newcommand{\Leb}{\mathcal{L}}
\newcommand{\bp}{\mathbf{p}}

\newcommand{\bv}{\mathbf{v}}
\newcommand{\bw}{\mathbf{w}}
\newcommand{\f}{\mathbf{f}}

\newcommand{\ve}{\varepsilon}

\newcommand{\ba}{\mathbf{a}}

\newcommand{\cc}{v}

\newcommand{\by}{\mathbf{y}}

\newcommand{\Q}{\mathbb{Q}}

\newcommand{\Z}{\mathbb{Z}}
\newcommand{\R}{\mathbb{R}}
\newcommand{\N}{\mathbb{N}}
\newcommand{\fff}{\mathit{f}}

\newcommand{\cM}{\mathcal{M}}

\newcommand{\ept}{(\varepsilon e^{-t})^{\frac{1}{2}}}
\newcommand{\eptm}{(\varepsilon e^{-t})^{-\frac{1}{2}}}

\newcommand{\ta}{\boldsymbol{\theta}}

\newcommand{\inv}{^{\text{-}1}}

\usepackage{mathtools}
\usepackage{tikz-cd}
\usepackage{graphicx}
\begin{document}

\title[Diophantine approximation on manifolds]{Rational points near manifolds and  Khintchine theorem}

\author{Victor Beresnevich}
\address{\textbf{Victor Beresnevich} \\
Department of Mathematics, University of York, UK }
\email{victor.beresnevich@york.ac.uk}
\author{Shreyasi Datta}
\address{\textbf{Shreyasi Datta} \\
Department of Mathematics, University of York, UK}
\email{shreyasi.datta@york.ac.uk}

\date{}

\maketitle
\begin{abstract}
In this paper we complete the long standing challenge to establish a Khintchine type theorem for arbitrary nondegenerate manifolds in $\R^n$. In particular, our main result finally removes the analyticity assumption from the Khintchine type theorem proved in [Ann. of Math. 175 (2012), 187-235]. Furthermore, we obtain a Jarn\`ik type refinement of our main result, which uses Hausdorff measures. The results are also obtained  in the  inhomogeneous setting. The proofs are underpinned by a sharper version of a `quantitative nondivergence' result of Bernik--Kleinbock--Margulis and a duality argument which we use to study rational points near manifolds.
\end{abstract}

\section{Introduction}\label{Intro}

Throughout $\ta\in\R^n$, $\psi:(0,\infty)\to (0,1)$ is a non-increasing function such that $\psi(q)\to 0$ as $q\to\infty$ and $\cL_n$ denotes Lebesgue measure on $\R^n$. In this paper we will be interested in the following set that naturally arises in a variety of problems in the theory of inhomogeneous Diophantine approximation:
\begin{equation}\label{Spsi}
\mathcal{S}_n^{\ta}(\psi):=\left\{\bx\in \R^n~:~\left\Vert \bx-\frac{\bp+\ta}{q}\right\Vert\leq \frac{\psi(q)}{q} \text{ for i.m. } (\bp,q)\in \Z^n\times \N\right\},
\end{equation}
where `i.m.' reads `infinitely many' and  $\Vert\cdot\Vert$ denotes the sup norm. The points $\bx$ lying in $\mathcal{S}_n^{\ta}(\psi)$ are often referred to as $(\psi, \ta)$-approximable. When $\ta=\bf0$, we deal with  homogeneous approximations to points $\bx$ in $\R^n$ by rational points $\bp/q$, in which case $\bx\in \mathcal{S}_n^{\bf0}(\psi)$ is called $\psi$-approximable. Since the set $\mathcal{S}_n^{\ta}(\psi)$ is invariant under translations by integer vectors, it is often restricted to the unit hyper-cube $[0,1]^n$.

By the inhimogeneous version of Khintchine's classical results, see \cite[\S1.1]{MR4537336}, we fully understand the  probability of a random point $\bx\in[0,1]^n$ being $(\psi,\ta)$-approximable: 

\medskip

\noindent\textbf{The Inhomogeneous Khintchine theorem:} {\em For any monotonic $\psi$ and $\ta\in\R^n$
\begin{equation}\label{eq1.2}
    \Leb_n(\mathcal{S}_n^{\ta}(\psi)\cap [0,1]^n)=\left\{\begin{aligned}
        &0 &\text{ if } \sum_{q=1}^\infty \psi(q)^n<\infty,\\
        &1 &\text{ if } \sum_{q=1}^\infty \psi(q)^n=\infty.
    \end{aligned}\right.
\end{equation}
}

\medskip

A longstanding  notorious problem in the theory of Diophantine approximation is to obtain Khintchine type results when $\bx$ is restricted to a subset of $\R^n$, such as a manifold or a fractal, which supports a natural (probability) measure. In particular, Khintchine type results have long been search for in the case of so-called nondegenerate manifolds. In this paper we complete this quest by establishing the following key result.

\begin{theorem}\label{thm: main0}
Let $n\ge2$, $\psi$ be non-increasing, $\ta\in\R^n$ and $\cM$ be any nondegenerate submanifold of $\R^n$. Then, almost no/almost every point $\bx\in\cM$ is $(\psi,\ta)$-approximable provided that the sum
\begin{equation}\label{K-sum}
\sum_{q=1}^\infty\psi(q)^n
\end{equation}
converges/diverges.  
\end{theorem}

\begin{remark}
Nondegenerate manifolds  satisfy a natural and mild geometric condition. While leaving the definition of nondegeneracy until \S\ref{sec1.1}, we now describe our key new ideas and how Theorem~\ref{thm: main0} builds upon previous findings. The convergence case of this theorem was previously established in \cite{BY} in the homogeneous setting, that is for $\ta=\vv0$, and thus the convergence case of Theorem~\ref{thm: main0} can be regarded as an extension of \cite{BY} to the inhomogeneous setting.
However, the divergence case of Theorem~\ref{thm: main0} is new, even for $\ta=\mathbf{0}$ -- previous divergence results did not reach the full class of nondegenerate manifolds even in the homogeneous case.  Indeed, for manifolds of dimension $d\ge2$ the divergence case has previously been known only for analytic nondegenerate manifolds -- this  is the main result of  \cite{Ber12}. Beyond analytic manifolds, the divergence case of Theorem~\ref{thm: main0} was only ever established for nondegenerate curves ($d=1$, $n\ge2$), which is the main result of \cite{BVVZ2018}. We note that both \cite{Ber12} and \cite{BVVZ2018} adopt a ubiquity approach relying on an application of quantitative non-divergence estimates first introduced by Kleinbock and Margulis in their landmark work \cite{KM}. It is the application of the quantitative non-divergence estimates that represents a challenge resulting in additional assumptions on nondegenerate manifolds in  \cite{Ber12} and \cite{BVVZ2018}. In this paper we present a new approach using a version of the quantitative non-divergence estimate of Bernik-Kleinbock-Margulis \cite{BKM} and an adaptation of a duality argument from \cite{BY}. In \S\ref{SQN}, we also establish a strengthening of the quantitative non-divergence estimate of \cite{BKM} which we use to obtain the Hausdorff measure version of Theorem~\ref{thm: main}, namely Theorem~\ref{main: Haus_div} stated below.
\end{remark}

\begin{remark}
Apart from papers  \cite{Ber12,BVVZ2018,BY} that we mentioned in the previous remark, over the years the search for a Khintchine type theorem for manifolds has yielded a number of other milestones. In this context, we ought to mention the case of planar curves ($n=2$, $d=1$), where a homogeneous Khintchine type theorem was first established for arbitrary $C^3$ nondegenerate curves in the case of convergence \cite{VV06} and divergence \cite{BDV07}. Subsequently, these were taken to the inhomogeneous setting in \cite{BVV14}, while the $C^3$ assumption was removed, in fact relaxed to weak nondegeneracy - a condition that is locally closed in the uniform convergence topology, in \cite{BZ10} and \cite{Hua15}. Previous results in higher dimensions, have been established for various subclasses of nondegenerate manifolds with various geometric restrictions which were excluding curves, see \cite{Ber12, BVVZ2017, Huang2020, MR4706444, SST, MR4404031} and references within.
\end{remark}

\begin{remark}
It should be noted that there is the \textit{dual} to \eqref{Spsi} setting, {\em e.g.} see \cite{MR2508636, BernikDodsonBook}, which even in the case of manifolds has been well understood for quite some time now. Indeed, the story begins with the pioneering work of Kleinbock and Margulis \cite{KM} addressing approximating functions of the form $\psi(t)=t^{-w}$. Subsequently, a full dual analogue of Khintchine's theorem for nondegenerate manifolds was obtained in \cite{Bere99, BBKM, BKM} for homogeneous approximations and in \cite{BBV, BV} for inhomogeneous approximations. We also refer the reader to \cite{BGGV2020, Ghosh05, Ghosh11, Kleinbock-extremal, Kleinbock-exponent} for relevant results for affine subspaces which represent the degenerate case.
\end{remark}

\subsection{Nondegeneracy of maps and manifolds}\label{sec1.1}

In view of the measure theoretic nature of the problems we investigate, it suffices to consider manifolds locally. Hence, without loss of generality we will assume that $\cM=\f(\bU)$, where $\bU$ is an open subset of $\R^d$, $1\le d<n$ and $\f:\bU\to\R^n$ is a nonsingular  $C^1$ map, that is the Jacobian matrix of $\f$ has rank $d$. Here and elsewhere $d$ stands for the dimension of $\cM$. Following \cite{KM}, we will say that a map $\f:\bU\to\R^n$ is $l$-nondegenerate at $\bx\in\bU$ if $\f$ is $C^l$ on a neighbourhood of $\bx$ and there are $n$ linearly independent partial derivatives of $\f$ of orders $\le l$. We will say that $\f$ is nondegenerate at $\bx$ if it is $l$-nondegenerate at $\bx$ for some $l$. We say that $\f$ is nondegenerate in $\bU$ if $\f$ is nondegenerate at (Lebesgue) almost every $\bx\in \bU$. Respectively, \ manifold $\cM$ is called nondegenerate if a neighborhood of (almost) every point can be parametrised by a  nondegenerate map. 
If we are given a (local) parameterisation $\f:\bU\to\R^n$ of $\cM$ so that  $\cM=\f(\bU)$, where $\bU$ is an open subset of $\R^d$, $d=\dim\cM$, then we will conveniently use the pushforward $\f_\star\Leb_d$ of $d$-dimensional Lebesgue measure. This is defined for $A\subset\cM$ as follows:  $\f_\star\Leb_d(A):=\Leb_d(\f^{-1}(A))$. In view of the above remarks, Theorem~\ref{thm: main0} can now be equivalently reformulated as follows:

\begin{theorem}\label{thm: main}
Let $n\ge2$, $\bU\subset\R^d$ be an open subset, $\f:\bU\to\R^n$ be nonsingular and nondegenerate in $\bU$. Let $\ta\in\R^n$ and $\psi$ be  non-increasing. Then 
\begin{equation}\label{eq1.2+}
\Leb_d\big(\f^{-1}\big(\mathcal{S}_n^{\ta}(\psi)\big)\big)=\left\{\begin{aligned}
        &0 &\text{ if } \sum_q \psi(q)^n<\infty,\\
        &\Leb_d(\bU) &\text{ if } \sum_q \psi(q)^n=\infty.
    \end{aligned}\right.
\end{equation}
\end{theorem}

\subsection{Rational points near manifolds}\label{intro_counting}
The key to the divergence result is the construction of an appropriate ubiquitous system of rational points lying near a given manifold. A by-product of this is a lower bound on the number of rational points lying close to a given manifold, namely a lower bound on
\begin{equation}\label{vb1}
\#\left\{(\bp,q)\in\Z^{n+1}~:~\inf_{\by\in \cM}\left\Vert \by-\frac{\bp+\ta}{q}\right\Vert\leq \frac{\psi(e^t)}{e^t},\; 0<q<e^t\right\},
\end{equation}
where $\mathcal{M}$ is a given manifold, or rather a generic open subset of a given manifold.

The convergence case generally requires an optimal (up to a constant factor) upper bound for
\eqref{vb1}. 
We note that as a consequence of the convergence of \eqref{K-sum}, the monotonic function $\psi$ must satisfy the inequality \begin{equation}\label{vb2}
\psi(t)<t^{-1/n}
\end{equation}
for all sufficiently large $t$.

Regarding \eqref{vb1}, in the groundbreaking work \cite{VV06} an optimal upper bound was established for any nonplanar $C^3$ curve in $\R^3$ for the best possible range of approximating functions $\psi$. But in higher dimensions obtaining the expected (heuristic) upper bound for \eqref{vb1} turns out to be a difficult problem. Indeed, until recently all the results in higher dimensions required extra conditions on the nondegenerate manifolds; see \cite{BVVZ2017,Huang2020}. In \cite{BY} for any nondegenerate manifolds a sharp upper bound was established only near a part of the manifold, that excludes a small Lebesgue measure set. More is known about the lower bound; see \cite{Ber12, BVVZ2018, SST}.

To state our new results we now introduce further notation. For any $\Delta\subset \bU,$
define 
$$N_{\ta}^\star(\Delta;\varepsilon, t):=\#\left\{ (\bp,q)\in\Z^{n+1}~|~ {e^{t-1}}\leq q\leq e^t, \inf_{\bx\in\Delta}\left\Vert \f(\bx)-\frac{\bp+\ta}{q}\right\Vert\leq \frac{\varepsilon}{e^t}\right\}$$ and 
$$
N_{\ta}(\Delta;\varepsilon, t):=\#\left\{ (\bp,q)\in\Z^{n+1}~|~ 1\leq q\leq e^t, \inf_{\bx\in\Delta}\left\Vert \f(\bx)-\frac{\bp+\ta}{q}\right\Vert\leq \frac{\varepsilon}{e^t}\right\}.
$$ 
Let 
\begin{equation}\label{def:eta}
\eta=\eta(n,d):=\begin{cases} \frac{1}{n-1} &\text{ if } d>1,\\[2ex]
    \frac{3}{2n-1} &\text{ if } d=1 .\end{cases}\end{equation}
By $x\ll_{\delta}y,$ we mean $x\leq c y,$ where $c>0$ is a constant that depends only on $\delta$, a parameter.
Now we state our main theorems on counting rational points near manifolds. 

\begin{theorem}[Lower bound]\label{mainthm: lower bound}
Let $n\ge2$, $ 1\le d<n$, $\ta\in\R^n$, $\bU\subset\R^d$ be an open subset, $\f:\bU\to\R^n$ be nonsingular and nondegenerate at  $\bx_0\in\bU$. Let  $\eta$ be as in \eqref{def:eta}. Then there exists an open neighborhood $B_0$ of $\bx_0$ such that for any ball $B\subset B_0$, there exists $t_0=t_0(B, n, \f)$ such that for every $t\geq t_0,$
\begin{equation}
    N_{\ta}^\star(B;\varepsilon, t)\gg_{\bx_0, d, \f} \varepsilon^m e^{(d+1)t}\Leb_d(B)
\end{equation}
for any $\varepsilon\in(0,1)$ satisfying 
\begin{equation}
\varepsilon\gg_{n, \f,\bx_0} e^{- t\eta}\,.
\end{equation}
\end{theorem}

Compared to previous results, Theorem~\ref{mainthm: lower bound} removes any extra conditions on the manifolds beyond the nondegeneracy. Indeed, the previous results; see \cite{Ber12, BVVZ2018, SST} and the references therein, have either dimension restriction or smoothness/analyticity assumptions on the manifolds, although their range of $\varepsilon$ may be bigger than that in Theorem~\ref{mainthm: lower bound}. Thus even for $\ta=\mathbf{0}$ Theorem~\ref{mainthm: lower bound} is new.

\begin{theorem}[Upper bound]\label{mainthm: upperbound}
   Let $n\ge2$, $ 1\le d<n, \ta\in\R^n, \bU\subset\R^d$ be an open subset, $\f:\bU\to\R^n$ be nonsingular and $l$-nondegenerate at  $\bx_0\in\bU$. Then there exists an open neighborhood $B_0$ of $\bx_0$ such that for some $t_0=t_0(B_0, n, \f)$, for every $t\geq t_0,$
 \begin{equation}\label{upper_bound}
        N_{\ta}^\star(B_0;\varepsilon,t)\ll_{n,\f, B_0}\varepsilon^m e^{(d+1)t}+ e^{(d+1)t} \left(\varepsilon^{n-1/2} e^{3t/2}\right)^{-\alpha},
    \end{equation} where $\alpha= \frac{1}{d(2l-1)(n+1)}$.
In particular, if $\varepsilon > e^{\frac{-3t}{2md(2l-1)(n+1)+2n-1}}$ then the first term in \eqref{upper_bound} dominates and the bound is optimal up to a constant factor.
\end{theorem}

Even in the case of $\ta=\mathbf{0},$ the above theorem extends \cite[Theorem 1.5]{SST}, where it was shown that if $\f$, as in Theorem~\ref{mainthm: upperbound}, is also $\lceil\frac{n+1}{\delta}\rceil$ times continuously differentiable for some $0<\delta\leq \frac{1}{2n+10}$, then $N^\star_{\mathbf{0}}(B_0; \varepsilon,t)\ll_{n,\f, B_0} \varepsilon^m e^{(d+1)t}$ when $\varepsilon > e^{\frac{-3t}{2md(2l-1)(n+1)+2n-1}} e^{\frac{(2n+1)\delta t}{2md(2l-1)(n+1)+2n-1}}.$ This is a strictly smaller range of $\varepsilon$ than the range of $\varepsilon$ in Theorem~\ref{mainthm: upperbound}.

\section{Hausdorff dimension and measure results}

In this section we provide refinements of the results stated in \S\ref{Intro} using Hausdorff measures and dimension. To begin with, we recall the following Hausdorff measure generalisation of Khintchine's theorem.

\smallskip

\noindent\textbf{The Inhomogeneous Khintchine-Jarn\'ik theorem:} {\em For any 
non-increasing $\psi$, any $\ta\in\R^n$ and any $0<s\leq n$}
\begin{equation}\label{eq1.1}
    \mathcal{H}^s(\mathcal{S}_n^{\ta}(\psi)\cap [0,1]^n)=\left\{\begin{aligned}
        &0 &\text{ if } \sum_{q=1}^\infty q^{n-s}\psi(q)^s<\infty,\\
        &\mathcal{H}^s([0,1]^n) &\text{ if } \sum_{q=1}^\infty q^{n-s}\psi(q)^s=\infty.
    \end{aligned}\right.
\end{equation}
Here and elsewhere $\mathcal{H}^s$ denotes $s$-dimensional Hausdorff measure. We refer the reader to \cite{Yann2004}, or \cite[Section 12.1]{BDV_memo} or  \cite[\S1.1\;\&\;1.3]{MR4537336} for a detailed historical account and further generalisations. The case of $s=n$ in \eqref{eq1.1} corresponds to the Inhomogeneous Khintchine theorem stated in \S\ref{Intro}. The following two results of this paper are Hausdorff measure generalisations of Theorem~\ref{thm: main}.

\begin{theorem}[Divergence case]\label{main: Haus_div}
Let $n\ge 2$, $\ta\in\R^n$, $\psi$ be a non-increasing function such that $\psi(q)^{1/\eta} q\to\infty$ as $q\to\infty$, where $\eta$ is given by \eqref{def:eta}. Let $d$, $\bU$ and $\f$ be as in Theorem~\ref{thm: main} and $d+m=n$. Then for any $\frac{(n+1)}{1+\eta}-m<s\leq d$, we have that 
     \begin{equation}
         \mathcal{H}^s\left( \f^{-1}(\mathcal{S}_n^{\ta}(\psi))\right)= \mathcal{H}^s(\bU) \text { if } \sum q^n \left(\frac{\psi(q)}{q}\right)^{s+m} =\infty.
     \end{equation}
     In particular, if $\tau(\psi):= \liminf_{q\to\infty} \frac{-\log\psi(q)}{\log q}$ satisfies $ \frac{1}{n}<\tau(\psi)<\eta$, then 
     \begin{equation}
         \dim_H\left(\f^{-1}(\mathcal{S}_n^{\ta}(\psi))\right)\geq \frac{n+1}{\tau(\psi)+1}-m.
     \end{equation}
 \end{theorem}
 
Note that the divergence case of Theorem~\ref{thm: main} is the $s=d$ case of Theorem~\ref{main: Haus_div}. 
\begin{remark}
For $d=1$, Theorem~\ref{main: Haus_div} was previously established in  \cite[Theorem 1.3]{BVVZ2018}, however our proof gives an alternate simpler argument. The main new substance of Theorem~\ref{main: Haus_div} is the case $d>1$ which extends \cite[Theorem~2.5]{Ber12} from analytic to any nondegenerate manifolds, although, when $d>1$, the range of $s$ (resp. the range of $\tau(\psi)$) we have in Theorem~\ref{main: Haus_div} is smaller than that in \cite[Theorem 2.5]{Ber12}.
\end{remark}

\begin{theorem}[Convergence case]\label{Thm: Hausdorff measure}
Let $n\ge2$, $\ta\in\R^n$, $\psi$ be a non-increasing function, $d$, $\bU$ and $\f$ be as in Theorem~\ref{thm: main} and $d+m=n$. Suppose that $s>0$,\begin{equation}\label{eqn: measure condition on degenerate}
        \mathcal{H}^s\{\bx\in \bU~:~ \f \text{  is not } l\text{-nondegenerate at } \bx\}=0\,,
    \end{equation}
    \begin{equation}\label{eqn: conv 2}
        \sum_{q\geq 1} q^n \left(\frac{\psi(q)}{q}\right)^{s+m}<\infty
    \end{equation} 
and 
    \begin{equation}\label{eqn: conv 3}
        \sum_{t\geq 1} \left(\frac{\psi(e^t
)}{e^{t}}\right)^{(s-d)/2} \left(\psi(e^t)^{n-1/2} e^{3t/2}\right)^{-\alpha}<\infty\,,
    \end{equation}
where
\begin{equation}\label{def:alpha}
\alpha=\frac{1}{d(2l-1)(n+1)}.
\end{equation}
Then \begin{equation}
       \mathcal{H}^s\left(\f^{-1}(\mathcal{S}_n^{\ta}(\psi))\right)=0.
    \end{equation} 
\end{theorem}
\begin{remark}
    It should be noted that even in the case of $\ta=\mathbf{0}$, the above theorem is new. It removes the smoothness condition in \cite[Theorem 1.11]{SST} and also improves the summation condition in \cite[Equation 1.15]{SST}. 
\end{remark}

Next, we have the following corollary that follows on combining Theorem~\ref{Thm: Hausdorff measure} and Theorem~\ref{main: Haus_div}. It improves the range of $\tau$ in \cite[Corollary 2.9]{BY}, removes the smoothness assumption from 
\cite[Corollary 1.13]{SST}, and extends these results to any $\ta\in\R^n$.   We write $ \mathcal{S}_{n}^{\ta}(\tau)$ for $\mathcal{S}_{n}^{\ta}(\psi)$ with $\psi(q)= q^{-\tau}.$

\begin{corollary}[Exact Hausdorff dimension]\label{exact}
    Let $n\geq 2$, $m$, $d$, $\f$ and $\alpha$ be as before, and  \begin{equation}
        \dim_{H}(\{\bx\in\bU~|~\f \text{  is not } l\text{-nondegenerate at } \bx\})\leq \frac{n+1}{\tau+1}-m,
    \end{equation}
where
    \begin{equation}\label{bad range of tau}
        \frac1n<\tau <  \frac{3\alpha+1}{\alpha(2n-1)+n}\,.
    \end{equation}
    Then 
    \begin{equation}
\dim_H(\f^{-1}(\mathcal{S}_{n}^{\ta}(\tau)))= \frac{n+1}{\tau+1}-m.
    \end{equation}
\end{corollary}

\subsection{Spectrum of inhomogeneous exponents}
We discuss the implications of our results for the inhomogeneous version of a problem by Bugeaud and Laurent regarding the Diophantine exponents of points on the Veronese curves $(x,x^2,\dots,x^n)$. Given an $x\in \R$ and $\ta\in\R^n$, first recall the definition of inhomogeneous exponents from \cite{BLmos}:
\begin{equation}
    \lambda_n^{\ta}(x):=\sup \{\tau~|~(x,x^2,\dots, x^n)\in \mathcal{S}_{n}^{\ta}(\tau)\}.
\end{equation}
Define 
$$\Spec^{\ta}(\lambda_n)=\left\{\lambda\in \left[0, \infty\right]~|~\lambda_{n}^{\ta}(x)=\lambda \text{ for some } x\in \R\right\}.$$ Note that when $\ta=\mathbf{0}$, $\lambda^{\mathbf{0}}_n(\frac{p}{q})=\infty$ for every $\frac{p}{q}\in\Q$ and  by Dirichlet's theorem $\lambda_n^{\mathbf{0}}(x)\geq 1/n$ for every $x\in \R.$ Thus the definition in \cite{BL05} and \cite{BY} matches with the above one when $\ta=\mathbf{0}.$ By Theorem~\ref{thm: main} we get that for almost every $x\in\R$, $\lambda_n^{\ta}(x)=\frac{1}{n}.$ Thus we know $\frac{1}{n}\in \Spec^{\ta}(\lambda_n).$ The following is a special case of Corollary \ref{exact} for $n\geq 2$, and $m=n-1$, $d=1$. Then for $\f(x)=(x,\dots, x^n),$
and $\tau$ satisfying 
    \begin{equation}
        \frac{1}{n}\leq \tau <  \frac{2n^2+n+2}{(n^2+n+1)(2n-1)},
    \end{equation}
  we have
    \begin{equation}
        \dim_H(\{\bx\in\bU~:~\f(\bx)\in \mathcal{S}_n^{\ta}(\tau)\})= \frac{n+1}{\tau+1}-n+1.
    \end{equation}
    This in particular gives the following corollary.
\begin{corollary}
     For every $n\geq 3,$
    \begin{equation}
        \left[ \frac{1}{n}, \frac{1}{n}+\frac{n+1}{n(2n-1)(n^2+n+1)}\right]\subset \Spec^{\ta}(\lambda_n).
    \end{equation}
\end{corollary}

\bigskip

\section{Sharp Quantitative Nondivergence}\label{SQN}

The proofs of the main result will rely on an effective upper bound for the measure of the set $\mathfrak{S}_{\f}(\delta,K,\bT)$ given by
\begin{equation}\label{eqn5.1}
\left\{\bx\in \bU:\exists\;(a_0,\ba)\in\Z\times\Z^n_{\neq\bf0}\;\;\text{such that }\left.
\begin{array}{l}
|a_0+\f(\bx)\ba^\top|<\delta\\[1ex]
\|\nabla \f(\bx)\ba^\top\|<K\\[1ex]
\vert a_i\vert<T_i ~~(1\leq i\leq n) 
\end{array}
\right.\right\},
\end{equation}
where $\ba=(a_1,\dots,a_n)$, $\bT=(T_1,\dots,T_n)$, $\delta,K,T_1,\cdots, T_n$ are positive real parameters, $\f:\bU\to\R^n$ is a map defined on an open set $\bU\subset\R^d$, 
$\nabla$ stands for the gradient of a real-valued function, $\Vert \cdot\Vert$ denotes the sup norm and $\ba^\top$ denotes the transpose of $\ba$. An upper bound on the measure of this set was already established in \cite{BKM}; here we provide an improved estimates which is crucial in order to get a larger range of exponents $\tau(\psi)$ in Theorem~\ref{main: Haus_div} than we could get otherwise. We prove the following theorem which is of independent interest.

\begin{theorem}\label{BKM_new}
Let $\bU\subset\R^d$ be open, $\bx_0\in\bU$ and $\f:\bU\to\R^n$ be $l$--nondegenerate at $\bx_0$. Write $n=m+d$ and suppose that
\begin{equation}\label{Monge0}
\f=(x_1,\dots,x_d,f_1(\bx),\dots,f_m(\bx)),\quad\text{where }\bx=(x_1,\dots,x_d).
\end{equation}
Then there exists a ball $B_0\subset\bU$ centred at $\bx_0$ and a constant $E>0$ depending on $B_0$ and $\f$ such that for any choice of $\delta,K,T_1,\cdots, T_n$ satisfying 
\begin{equation}\label{eqn5.2}
0<\delta\le 1,\qquad T_1,\dots,T_n\ge1,\qquad K>0 \quad\text{and}\quad \delta^{n}<K \frac{T_1\cdots T_n}{\max_i T_i}
\end{equation}
and any ball $B\subset B_0$ of radius $0<r\le1$ we have that
\begin{equation}\label{eqn5.3}
\begin{array}[b]{l}
\displaystyle\rule{0ex}{3ex}\Leb_d\big(\mathfrak{S}_{\f}(\delta,K,\bT)\cap B\big)\ll\;\\[1ex] \displaystyle\ll \delta \prod_{i=1}^d\min\{K,T_i\}\cdot(T_{d+1}\cdots T_n)\Leb_d(B)+ E\left(\delta \min\{K,r^{-1}\}\frac{T_1\cdots T_n}{\max_i T_i}\right)^{\alpha},
\end{array}
\end{equation}
where $\alpha$ is given by \eqref{def:alpha} and the implied constant depends on $l$, $m$ and $d$ only.
\end{theorem}

Theorem~\ref{BKM_new} can be regarded as a refinement of the following simplified version of Theorem~1.4 from \cite{BKM}, which is also used in \cite{BY} in this form.

\begin{theorem}[{\cite[Theorem~1.4]{BKM}}]\label{thm:KM}
Let $\bU\subset\R^d$ be open, $\bx_0\in\bU$ and $\f:\bU\to\R^n$ be $l$--nondegenerate at $\bx_0$. Then there exists a ball $B_0\subset\bU$ centred at $\bx_0$ and a constant $E\ge1$ such that for any choice of $\delta,K,T_1,\cdots,T_n$ satisfying \eqref{eqn5.2}
we have that
\begin{equation}\label{eqn5.4}
\Leb_d\big(\mathfrak{S}_{\f}(\delta,K,\bT)\cap B_0\big)\;\le\;E \left(\delta K\frac{T_1\cdots T_n}{\max_i T_i}\right)^{\alpha}\Leb_d(B_0)\,,
\end{equation}
where $\alpha$ is given by \eqref{def:alpha}.
\end{theorem}

We can say that the constant $E$ appearing in \eqref{eqn5.3} is related but may be different to the constant $E$ appearing in \eqref{eqn5.4}; see the details in the proof of Theorem \ref{BKM_new}. We note that Theorem~\ref{thm:KM} together with the following statement appearing as Theorem~1.3 in \cite{BKM} is already enough to prove a weaker version of Theorem~\ref{BKM_new} in which one would have $\min\{K,T^{1/2}\}$ with $T=\max_{1\le i\le n} T_i$ in place of the term $\min\{K,r^{-1}\}$ in \eqref{eqn5.3}.

\begin{theorem}[{\cite[Theorem~1.3]{BKM}}]\label{thm:KM2}
Let $\bU\subset\R^d$ be open, $\f:\bU\to\R^n$ be a $C^2$ map.
Let $B$ be a ball of radius $r>0$ such that $2B\subset\bU$, and let
\begin{equation}\label{M_BKM}
L:=\sup_{\substack{\bx\in 2B\\ |\beta|=2}}\|\partial_\beta\f(\bx)\|<\infty\,.
\end{equation}
Then, for any $\ba\in\R^n$ such that
\begin{equation}\label{eq5.5}
\|\ba\|\ge\frac{1}{4nLr^2}
\end{equation}
the set
\begin{equation}
\mathfrak{S}'_{\f}(\delta,\ba):=\left\{\bx\in B:\exists\;a_0\in\Z\;\text{such that}\;\left.\begin{array}{l}
     |a_0+\f(\bx)\ba^\top|<\delta  \\
     \|\nabla\f(\bx)\ba^\top\|\ge\sqrt{ndL\|\ba\|}
\end{array}\right.\right\}
\end{equation}
has measure at most $C_d\delta\Leb_d(B)$, where $C_d>0$ is a constant depending on $d$ only.
\end{theorem}

We note that Theorem~1.3 in
\cite{BKM} is stated for $\ba\in\Z^n$, however nowhere is the proof does it use the fact that $\ba$ is integer, and assuming that $\ba\in\R^n$ is enough.
The proof of Theorem~\ref{BKM_new} will use the following strengthening of Theorem~\ref{thm:KM2}.

\begin{theorem}\label{thm:BKM+}
Let $\bU\subset\R^d$ be open, $\bx_0\in\bU$ and $\f:\bU\to\R^n$ be $l$--nondegenerate at $\bx_0$. Then there exists a ball $B_0$ centred at $\bx_0$ and satisfying $2B_0\subset\bU$, and a constant $C_1>0$ depending on $\bx_0$ such that for any ball $B\subset B_0$ of radius $r>0$ and any $\ba\in\R^n$ satisfying 
\begin{equation}\label{eq5.5+}
\|\ba\|\ge\max\left\{\frac{1}{4nLr^2},\frac{ndL}{C_1^2}\right\}\,,
\end{equation}
where $L$ is given by \eqref{M_BKM}, the set
\begin{equation}\label{eqn5.5}
\mathfrak{S}''_{\f}(\delta,\ba):=\left\{\bx\in \bU:\exists\;a_0\in\Z\;\text{such that}\;\left.\begin{array}{l}
     |a_0+\f(\bx)\ba^\top|<\delta  \\
     \|\nabla\f(\bx)\ba^\top\|\ge r^{-1}
\end{array}\right.\right\}
\end{equation}
satisfies 
$$
\Leb_d(\mathfrak{S}''_{\f}(\delta,\ba)\cap B)\le C_{l,d}\delta\Leb_d(B)\,,
$$
where $C_{l,d}>0$ is a constant depending on $l$ and $d$ only.
\end{theorem}

We remark that since $\f$ is non-degenerate at $\bx_0$, we necessarily have that $\f\in C^2$ on a neighbourhood of $\bx_0$. This means that the constant $L$, given by \eqref{M_BKM}, is well defined. Furthermore, we can assume that $L$ is non-zero.

The proof of Theorem~\ref{thm:BKM+} will use the following proposition which extends \cite[Proposition 1]{Ber99Acta}.

\begin{proposition}\label{prop5.1}
Let $0<\delta<1$, $k\ge2$, $\Theta\ge1/2$, $I\subset\R$ be an interval and $F:I\to\R$ be a $C^k$ function such that $\inf_{x\in I}|F^{(k)}(x)|>0$. Let $\mathfrak{S}^{(1)}(F,\delta,\Theta,I)$ be the set of $x\in I$ such that
\begin{equation}\label{eqn5.6}
  |F(x)+a|<\delta,\qquad |F'(x)|> \Theta^{-1}\Leb_1(I)^{-1}
\end{equation}for some $a\in\Z$. 
Then
\begin{equation}\label{eqn5.7}
\Leb_1(\mathfrak{S}^{(1)}(F,\delta,\Theta,I))\le 8k\Theta\delta \Leb_1(I)\,.
\end{equation}
\end{proposition}

\begin{proof}
Without loss of generality, we will assume that $\delta<\tfrac14$ as otherwise \eqref{eqn5.7} is trivially true. Let $\sigma(a)$ be the set of $x\in I$ satisfying \eqref{eqn5.6} for a given $a$. Clearly, $\sigma(a)$ can be written as a union of non-empty disjoint intervals $\sigma(a)=\cup_i\sigma_i(a)$. In what follows we will assume without loss of generality that this union is minimal, that is $\sigma_i(a)\cup \sigma_j(a)$ is not an interval if $i\neq j$. Fix any $\ve>0$ and  for each choice of $a$ and $i$, fix any $\alpha_{a,i}\in \sigma_i(a)$ such that
\begin{equation}\label{eqn5.8}
\inf_{x\in \sigma_i(a)}|F'(x)|\ge \frac{|F'(\alpha_{a,i})|}{1+\ve}\,.
\end{equation}
By the Mean Value Theorem, \eqref{eqn5.6} and \eqref{eqn5.8}, $\Leb_1(\sigma_i(a))\le 2(1+\ve)\delta|F'(\alpha_{a,i})|^{-1}$. Therefore,
\begin{equation}\label{eqn5.9}
\Leb_1(\mathfrak{S}^{(1)}(F,\delta,I))\le\sum_{a}\sum_i \Leb_1(\sigma_i(a))\le 2(1+\ve)\delta\sum_{a}\sum_i|F'(\alpha_{a,i})|^{-1}\,.
\end{equation}
Since $\inf_{x\in I}|F^{(k)}(x)|>0$, $F'$ and $F''$ have $s\le(k-1)+(k-2)=2k-3$ zeros inside $I$. Thus $I$ can be written as a disjoint union of subintervals $I_0,\dots,I_s$ such that the above zeros are one of the endpoints of these subintervals. For each $j$, let $A_j$ be the set of all $\alpha_{a,i}\in I_j$. Thus $\cup_{a}\cup_{i}\{\alpha_{a,i}\}=\cup_j A_j.$ Note that on each interval $I_j$, $F'$ is monotonic and does not change sign. For simplicity we will assume that $F'$ is increasing (the decreasing case is considered exactly the same way) and positive. Therefore $F$ is increasing. Then, we claim that if $\alpha_{a_1,i_1}\le\alpha_{a_2,i_2}$ are two consecutive elements in $I_j$ then we must have that $a_1\neq a_2$. To prove this, assume for the moment that the claim is not true and let $\alpha_{a,i_1}\le\alpha_{a,i_2}$ be two consecutive elements in $I_j$. Let $c_1$ be the right most end point of $\sigma_{i_1}(a)$, $c_2$ be the left most end point of $\sigma_{i_2}(a).$ By the minimality of the union $\sigma(a)=\cup_i\sigma_i(a)$, we must have that $c_1<c_2$ and $\vert F(c_i)+a\vert=\delta$ for $i=1,2.$ This implies that $F'$ must vanish between $c_1$ and $c_2$, contrary to the fact that $F'$ is not zero inside $I_j$, thus proving the claim. Next, using the inequalities $\delta<\tfrac14$ and $|F(\alpha_{a_\ell,i_\ell})+a_\ell|<\delta$ for $\ell=1,2$, we obtain that
$$
\tfrac12\le|a_1-a_2|-2\delta\le |F(\alpha_{a_1,i_1})-F(\alpha_{a_2,i_2})|\,.
$$
Then, using the Mean Value Theorem and the assumption that $F'$ is increasing, we get that
$$
\tfrac12\le|F(\alpha_{a_1,i_1})-F(\alpha_{a_2,i_2})|\le |F'(\alpha_{a_2,i_2})|(\alpha_{a_2,i_2}-\alpha_{a_1,i_1})\,,
$$
whence
$$
|F'(\alpha_{a_2,i_2})|^{-1}\le 2(\alpha_{a_2,i_2}-\alpha_{a_1,i_1})\,.
$$
On summing the latter over $\alpha_{a_2,i_2}\in I_j$ and using the fact that the sum of right hand side is $\le 2\Leb_1(I_j)$, we get that
$$
\sum_{\alpha_{a,i}\in A_j}|F'(\alpha_{a,i})|^{-1}\le 2\Leb_1(I_j)+\max_{\alpha_{a,i}\in A_j}|F'(\alpha_{a,i})|^{-1}
\stackrel{\eqref{eqn5.6}}{\le} 2\Leb_1(I_j)+\Theta \Leb_1(I)\,.
$$
Finally, since the number of different intervals $I_j$ is $\le 2k-2$ and $\sum_j \Leb_1(I_j)=\Leb_1(I)$, we get that
$$
\sum_{a}\sum_i|F'(\alpha_{a,i})|^{-1}\le 2 \Leb_1(I)+(2k-2)\Theta\Leb_1(I)\le (2k+2)\Theta \Leb_1(I)\,,
$$ since $\Theta>\frac{1}{2}.$
Combining this with \eqref{eqn5.9} and letting $\ve\to0$ gives the required bound.
\end{proof}

\medskip

\begin{proof}[Proof of Theorem~\ref{thm:BKM+}]
Since $\f$ is $l$-non-degenerate at $\bx_0$, there is a constant $C_0>0$ such that for any $\ba\in\R^n$ with $\|\ba\|\ge1$ there exists a multi-index $\beta$ with $1\le|\beta|=k\le l$ such that
$$
|\partial_\beta\f(\bx_0)\ba^\top|\ge C_0 \|\ba\|\,.
$$
This can be explained as follows. First, let us consider $$\mathcal{F}:=\{g=\f(\cdot)\ba^\top~|~\ba\in\R^n, \Vert \ba\Vert=1\}.$$
By the compactness of $\mathcal{F}$ and the nondegeneracy of $\f$ at $\bx_0$, we have that
$$\inf_{g\in\mathcal{F}}\max_{\vert\beta\vert\leq l}\vert \partial_{\beta}g(\bx_0)\vert$$ is attained at some $g$ and is nonzero. Thus we can take $C_0= \inf_{g\in\mathcal{F}}\max_{\vert\beta\vert\leq l}\vert \partial_{\beta}g(\bx_0)\vert. $ The subsequent part of the proof will use the following 

\medskip

\noindent\textbf{Claim:} There exists a rotation of the coordinate systems in $\R^d$ around $\bx_0$, which may depend on $\ba$, such  that $|\tilde\partial_i^k\f(\bx_0)\ba^\top|\ge 2C_1\|\ba\|$ for $i=1,\dots,d$ for some $C_1>0$ independent of $\ba$, where $\tilde\partial_i^k$ refers to derivation in the rotated coordinate system.

\medskip

For completeness we will provide a proof of this claim in \S\ref{proof-of_claim}. For now we continue with the proof of Theorem~\ref{thm:BKM+}.
By continuity, there is an open neighbourhood $\bU_0\subset\bU$ of $\bx_0$ such that for every $\ba\in\R^n\setminus\{\vv0\}$ there is an appropriate rotation of the coordinate systems around $\bx_0$ such that 
\begin{equation}\label{eqn5.10}
  |\tilde\partial_i^k\f(\bx)\ba^\top|\ge C_1\|\ba\|\quad\text{for $i=1,\dots,d$, and for all $\bx\in \bU_0$}.
\end{equation}
Now let $B_0$ be a (Euclidean) ball centred at $\bx_0$ such that whenever a ball $B$ lies in $B_0$, any cube $\hat B$ circumscribed around $B$ is contained in $\bU_0$.

Fix any ball $B\subset B_0$ and let $r$ denote the radius of $B$. Fix any $\ba\in\R^n$  with $\|\ba\|\ge1$ and let $\hat B$ be the 
cube circumscribed around $B$ which is aligned with the rotated coordinate system for which \eqref{eqn5.10} holds. By the choice of $B_0$ we have that $\hat B\subset \bU_0$.
Clearly,
\begin{equation}\label{eq5.11}
\Leb_d(\mathfrak{S}''_{\f}(\delta,\ba)\cap B)\le
\Leb_d(\mathfrak{S}''_{\f}(\delta,\ba)\cap \hat B)\,.
\end{equation} 
Since the rotation of the coordinate system rotates the gradient, the second inequality in \eqref{eqn5.5} implies that $\|\tilde\nabla \f(\bx)\ba^\top\|\ge C_2 r^{-1}$ in the new coordinate system, where $0<C_2\le1$ depends on $d$ only. The condition on the gradient means that for some $1\le i\le d$ we have that $|\tilde\partial_i\f(\bx)\ba^t|\ge C_2r^{-1}$. 

\medskip

Now, with the view to using Proposition~\ref{prop5.1} we perform a reduction to dimension one. Indeed,
by Fubini's theorem, the right hand side of \eqref{eq5.11} is bounded by 
\begin{equation}\label{eq5.11b}
d(2r)^{d-1}\sup_{\bx'} \Leb_1(\mathfrak{S}^{(1)}(F_{\bx'},\delta,\Theta,I))\,,
\end{equation}
where $\bx'$ is the projection of $\bx=(x_1,\dots, x_d)$ obtained by removing $x_i$ for some $i$ and $F_{\bx'}(x_i)=\f(x_1,x_2,\dots,x_d)\ba^\top$ is defined on an interval $I$ of length $\Leb_1(I)=2r$ and $\Theta=C_2^{-1}$. Suppose that $k\ge2$. Then, by Proposition~\ref{prop5.1} and the fact that $k\le l$, we get that \eqref{eq5.11b} is bounded by $8dlC_2^{-1}\delta (2r)^{d}$, which implies the required result.

It remains to consider the case $k=1$. First, note that by \eqref{eqn5.10}, we have $\Vert \tilde\nabla\f(\bx)\ba^\top\Vert \geq C_1 \Vert \ba\Vert^{1/2} \Vert \ba\Vert^{1/2}\stackrel{\eqref{eq5.5+}}{\geq} C_1 \frac{\sqrt{ndL\Vert\ba\Vert}}{C_1}\geq r^{-1}$ for all $\bx\in \hat{B}$. Thus the second inequality in the definition of $\mathfrak{S}''_{\f}(\delta,\ba)\cap \hat B$ is redundant, and moreover using \eqref{eq5.5+} the aforementioned set is the same as $\mathfrak{S}'_{\f}(\delta,\ba)\cap \hat B$. Thus, we apply Theorem~\ref{thm:KM2} to complete the proof. 
\end{proof}

\medskip

\begin{proof}[Proof of Theorem~\ref{BKM_new}]
If $K<r^{-1}$ then $\min\{K,r^{-1}\}=K$ and \eqref{eqn5.3} becomes an immediate consequence of Theorem~\ref{thm:KM}.
Thus for the rest of the proof we will assume that $K\ge r^{-1}$. In this case, since $\delta^n\leq 1\leq r^{-1} \frac{T_1\cdots T_n}{\max_i T_i}$, by Theorem~\ref{thm:KM} again we have that
\begin{equation}\label{eqn5.11}
\Leb_d\big(\mathfrak{S}_{\f}(\delta,r^{-1},\bT)\big)\;\ll E\left(\delta \min\{K,r^{-1}\}\frac{T_1\cdots T_n}{\max_i T_i}\right)^{\alpha}
\end{equation}
while
\begin{equation}\label{eqn5.12}
\mathfrak{S}_{\f}(\delta,K,\bT)\setminus \mathfrak{S}_{\f}(\delta,r^{-1},\bT)\subset\bigcup_{\ba}\mathfrak{S}''_{\f}(\delta,\ba)\,,
\end{equation}
where the union is taken over all $\ba=(a_1,\cdots,a_n)\in\Z^n$ such that for some $\bx\in B_0$
\begin{equation}\label{eqn5.13}
\|\nabla \f(\bx)\ba^\top\|<K\quad\text{and}\quad\vert a_i\vert\le T_i\quad (1\le i\le n)\,.
\end{equation}
By the right hand side of \eqref{eqn5.13}, there are at most $\ll T_{d+1}\cdots T_n$ choices for $a_{d+1},\dots,a_n$. In turn, in view of \eqref{Monge}, by the left and right hand sides of \eqref{eqn5.13}, for each fixed choice of $a_{d+1},\dots,a_n$, there are at most $\ll \prod_{i=1}^d\min\{K,T_i\}$ choices for $a_1,\dots,a_d$. Thus, the union in \eqref{eqn5.12} contains  $\ll \prod_{i=1}^d\min\{K,T_i\}T_{d+1}\cdots T_{n}$ sets. In the case \eqref{eq5.5+} is satisfied, Theorem~\ref{thm:BKM+} is applicable and the corresponding contribution from the union in \eqref{eqn5.12} coincides with the first term of the estimate in \eqref{eqn5.3}. Finally, the union in the right of \eqref{eqn5.12} taken over $\ba$ such that
\begin{equation}\label{s}
\|\ba\|\le\max\left\{\frac{1}{4nLr^2},\frac{ndL}{C_1^2}\right\}
\end{equation} is obviously contained in
$\mathfrak{S}_{\f}(\delta,T_0,T_0,\cdots, T_0)$
with $$
T_0= C_0\max\left\{1,\frac{1}{4nLr^2},\frac{ndL}{C_1^2}\right\}\ge1
$$ 
for a suitably chosen constant $C_0$ depending on $\f$ only. In this last case we again apply Theorem~\ref{thm:KM}, and obtain an upper bound $\ll_{\f, B} E\delta^{\alpha}\leq E\left(\delta \min\{K,r^{-1}\}\frac{T_1\cdots T_n}{\max_i T_i}\right)^{\alpha}$ since $K\geq r^{-1}\geq 1$.
\end{proof}

\subsection{Proof of the Claim}\label{proof-of_claim}

Here we give a proof of the Claim stated within the proof of Theorem~\ref{thm:BKM+}. This will be done in two steps: first we show the existence of a suitable constant for a given $\ba$, then we prove that the constant can be chosen to be independent of $\ba$. 

\medskip

\noindent\textbf{Step 1:} Let us fix $\ba\in\R^n\setminus \{\mathbf{0}\}$. First, we show that there exists a rotation of coordinate system (depending on $\ba$) around $\bx_0$ such that $|\tilde\partial_i^k\f(\bx_0)\ba^\top|\neq 0$ for every $i=1,\dots, d.$ Here $\tilde\partial$ denotes differentiation in the new coordinate system. To this end, given $A=(a_{ij})\in \mathrm{GL}_d(\R)$, differentiating the map $g:=\f(\cdot)\ba^\top$ we have, by the chain rule, that
              	\begin{equation*}\label{lin_sys1}
              	\begin{array}{rcr}
              	\partial_{1}^{k}g\circ A(A \inv \bx_0) &=\sum_{\alpha=(i_1,\dots, i_d), \vert \alpha\vert =k}  a_{11}^{i_1}\cdots a_{d1}^{i_d} \ \partial_{\alpha} g(\bx_0)\,, \\
              		
  \vdots\qquad\quad       \\
              		
             		\partial_{d}^{k}g\circ A(A\inv \bx_0) &= \sum_{\alpha=(i_1,\dots,i_d), \vert \alpha\vert =k}  a_{1d}^{i_1}\cdots a_{dd}^{i_d} \ \partial_{\alpha} g(\bx_0)\,. 
             		\end{array}
              		\end{equation*}
We wish to find  $A\in \mathrm{O}_n(\R) $ such that $x'_i\neq 0$ for all $i=1,\dots, d$ where

$$\begin{array}{rcr}
              x'_1&=& P(a_{11},\dots,a_{d1})\,,  \\
              \vdots\; \\
              x'_d&=& P(a_{1d},\dots,a_{dd}) \,,
              \end{array}
              $$ and  
              $$P(x_1,\dots, x_d)= \sum_{i_1+\cdots+i_d=k}  x^0_{(i_1,\dots,i_d)} x_1^{i_1}\cdots x_d^{i_d},$$ and $x^0_{(i_1,\dots,i_d)}=\partial_{(i_1,\dots, i_d)} g(\bx_0).$
              Note that $P$ is a homogeneous polynomial of degree $k$. Since $ \partial_{ \beta} g(\bx_0)\neq 0$, i.e., $x^0_{\beta}\neq 0$, $P $ is a nonzero polynomial. Also, since $P$ is homogeneous, if $P(\by)\neq 0$ for some $\by=(y_1,\dots, y_d)\neq \mathbf{0},$ we also have $P(\by')\neq 0$ for $\by'=\by/\Vert\by\Vert.$
              
We prove by induction that for any $s\leq d$ we can construct nonzero orthonormal vectors $\bv_1,\dots, \bv_s$ such that $\prod_{i=1}^s P(\bv_i)\neq 0.$ When $s=1$, the statement follows since $P$ is a nonzero and homogeneous. Suppose $\bv_1, \cdots, \bv_s, s<d$ are constructed such that $\prod_{i=1}^{s} P(\bv_i)\neq 0.$ Since $P$ is nonzero at each of these vectors, there exists $W_i, i=1,\dots,s$ open balls centered at $\bv_i$ of radius $r_i$ such that $P$ is nonzero at each point in $W_i$ for $i=1,\dots,s$. Let $\mathcal{L}$ be the subspace that is perpendicular to each $\bv_i.$ Let us take any unit vector (length $1$) $\bv\in \mathcal{L}$, and $V:=\{cO\bv~|~\Vert O-\mathrm{I}\Vert<r, O\in\mathrm{O}(n), c>0\}$ an open neighborhood of $\bv$, where $r<\min r_i.$ Since $P$ is nonzero polynomial, there exists a unit vector $\bw\in V$ such that $P(\bw)\neq 0.$ Thus $\bw=O\bv$ for some $O\in \mathrm{O}(n)$ with $\Vert O-\mathrm{I}\Vert<r$. Then for $i=1,\dots,s$, $\Vert O\bv_i-\bv_i\Vert<r<r_i$, hence $\bv_i'=O\bv_i\in W_i$. By the construction of $W_i$, $P(\bv_i')\neq 0$ for $i=1,\dots, s.$ Since $\bv_1,\dots, \bv_s,\bv$ are orthonormal, $\bv_1',\dots,\bv_s',\bw$ are also orthonormal. 
Let us denote $\bw=\bv_{s+1}'$. Thus we constructed orthonormal vectors $\bv_i', i=1,\dots, s+1$ such that $\prod_{i=1}^{s+1} P(\bv_i')\neq 0.$ 

To summarize the above, we have that 
for every $\ba\in\R^n\setminus \{0\},$ there exists $O_{\ba}\in O_n(\R)$ such that $\tilde\partial_{i}^k \f(\bx_0)\ba^\top:=\partial_{i}^k \f\circ O_{\ba}(O_{\ba}^{-1} \bx_0)\ba^\top\neq 0,$ for all $i=1,\dots,d.$ This implies 
\begin{equation}\label{eqn:nonzero}\min_{i=1}^d\vert \partial_{i}^k \f\circ O_{\ba}(O_{\ba}^{-1} \bx_0)\ba^\top\vert>0.\end{equation} This concludes the first part of the claim.

\medskip
          
\noindent\textbf{Step 2:} Finally, we show the existence of $C_1>0$ as in the claim. Let 
$$C_1:=\inf_{\Vert\ba\Vert=1} \sup_{O\in O_n(\R)} \min_{i=1}^d\vert\partial_{i}^k \f\circ O(O^{-1} \bx_0)\ba^\top\vert.$$ 
The goal is to show that  $C_1>0.$
We prove this by contradiction. If $C_1=0$, then there exists a  sequence $\ba_m$ with $\Vert \ba_m\Vert=1$ such that $$\sup_{O\in O_n(\R)} \min_{i=1}^d\vert\partial_{i}^k \f\circ O(O^{-1} \bx_0)\ba_m^\top\vert<\frac{1}{m}.$$ By the compactness of the unit sphere $\{\ba\in\R^n \;|\; \Vert\ba\Vert=1\}$, there exists a convergent subsequence $\ba_{m_l}\to \ba$, where $\Vert \ba\Vert=1$, such that  
$$\sup_{O\in O_n(\R)} \min_{i=1}^d\vert\partial_{i}^k \f\circ O(O^{-1} \bx_0)\ba^\top\vert=0.$$ This contradicts \eqref{eqn:nonzero} and completes the proof of the claim.

\section{Auxiliary statements and definitions}

\subsection{Some geometry of numbers} Before moving on to the proofs of the main results, we recall some auxiliary facts about lattices. In what follows, given a lattice $\Lambda$ in $\R^{n+1}$, 
$$
\delta_1(\Lambda),\dots,\delta_{n+1}(\Lambda)
$$
will denote Minkowski's {\em successive minima} of the closed unit Euclidean ball centered at $\mathbf{0}$ on the lattice $\Lambda$.
For $g\in \GL_{n+1}(\R),$ by Minkowski's second theorem, we have that 
$$\delta_{1}(g\Z^{n+1}) \cdots \delta_{n+1}(g\Z^{n+1})\asymp_n \vert\det g\vert\,.
$$
In particular,
\begin{equation}\label{eqn: Mincon}
\frac{\vert \det g\vert}{ \delta_1(g\Z^{n+1})}\ll_n \delta_{n+1}(g\Z^{n+1})^n.\end{equation}
For any $g\in \GL_{n+1}(\R),$ we define the dual matrix $g^\star=\sigma^{-1}(g^{\top})^{-1}\sigma$. 
Here and elsewhere, $A^\top$ denotes the transpose of the matrix/row/column $A$, $\sigma=\sigma_{n+1}$ and 
$$
\sigma_k = \begin{bmatrix}
  0 &  \dots & 1 \\
  \vdots  & \iddots  & \vdots \\
  1 &  \dots & 0
\end{bmatrix}
$$ is the long Weyl element of $\GL_k(\R)$, where $k$ is the size of the matrix.

We state the following lemma that appears in \cite[Lemma 3.3]{BY}. We note that it was stated for $g\in\SL_{n+1}(\R)$ in \cite{BY} but it is true for $g\in\GL_{n+1}(\R)$ for the same reasons  given in \cite{BY}, in particular by Theorems~21.5~and~23.2 from \cite{MR2335496}. 

\begin{lemma}\label{dual lemma}
    Let $g\in \mathrm{GL}_{n+1}(\R)$ then 
    \begin{equation}
        \delta_1(g\Z^{n+1}) \delta_{n+1}(g^\star \Z^{n+1}) \asymp_{n} 1. 
    \end{equation}
\end{lemma}

\subsection{Further assumptions and notation}
Write $n=d+m$ with $d,m\in\Z$. We will assume throughout that $1\le d<n$ as otherwise the manifold is essentially $\R^n$.
As we already mentioned before, in view of the nonsingularity assumption on $\f$, by the Implicit Function Theorem, there will also be no loss of generality in assuming that $\f$ has the following form 
\begin{equation}\label{Monge}
\f=(\bx,f(\bx)),\quad\text{where }f(\bx)=(f_1(\bx),\dots, f_m(\bx))\text{ and } \bx=(x_1,\dots,x_d)\,.
\end{equation} Since $d<n$ and $\f$ is nondegenerate, we must have that $\f$ is at least $C^2$ in the context of our theorems.  Furthermore, by shrinking $\bU$ if necessary, we can assume without loss of generality that for some constant $M\ge 1$
 \begin{equation}\label{condi2}
    \max_{1\leq k\leq m}\max_{1\leq i,j\leq d}\sup_{\bx\in\bU}\max \left\{\vert {\partial_i\mathit{f}_k}(\bx)\vert,\vert {\partial^2_{i,j}\mathit{f}_k}(\bx)\vert\right\} \leq M.
 \end{equation}
For a vector $\by\in\R^n$ and  $1\leq i\leq n$, we denote the $i$-th coordinate of $\by$ by $y_i$.  Given $\bx\in\R^k$ for some $k$, by $B(\bx,r)$ (resp. $B[\bx,r]$), we denote the open ball (resp. closed ball) in $\R^k$ centred at $\bx$ of radius $r$, which is obtained using the sup norm on $\R^k$. By $\Vert \cdot\Vert$, we denote the sup norm in $\R^k.$ By $[~\cdot~]^k$ we denote the hyper-cube (i.e. ball with sup norm) in $\R^k$ of the specified radius and centred at $\bf0$. For example, $[\pi]^k$ will stand for $[-\pi,\pi]^k$.

Throughout, $\mathrm{I}_k$ will denote the identity matrix of $\GL_k(\R)$. Now we recall some matrices introduced in \cite{BY}. For every $\bx\in\bU$ define
$$
J(\bx)=\big[\partial_i\fff_j(\bx)\big]_{1\leq i\leq d,\, 1\leq j \leq m}\,, $$ 
which is the Jacobian of the map $\bx\to f(\bx)=(f_1(\bx),\dots,f_m(\bx))$.
Further, following \cite{BY} for any given $\bx\in\bU$ define the following elements of $\SL_{n+1}(\R)$:
\begin{align*}
&z(\bx)=\begin{bmatrix}
     \mathrm{I}_m & -\sigma_m^{-1}J(\bx)\sigma_d & 0\\
     0 & \mathrm{I}_d & 0\\
     0 & 0 & 1
\end{bmatrix},\\[2ex]
&u({\bx})=\begin{bmatrix}
     \mathrm{I}_n & \sigma_n^{-1}\f(\bx)^\top\\
     \mathbf{0} & 1
\end{bmatrix}\,,\\[2ex]
&u_1(\bx)=z(\bx) u(\bx)\,.
\end{align*}
It can be easily seen that $u_1(x)$ is the following upper triangular matrix:\begin{equation}\label{def_u_1}
u_1(\bx)=\begin{bmatrix}
    1 & 0 & \cdots & 0 & -\partial_d f_m(\bx) & \cdots &  -\partial_1 f_m(\bx) & f_{m}(\bx)- \sum_{i=1}^d x_i\partial_i f_m(\bx)\\
    \vdots & \vdots & \vdots & \vdots & \vdots & \vdots & \vdots & \vdots \\
    0 & 0 & \cdots & 1 & -\partial_d f_1(\bx) & \cdots &  -\partial_1 f_1(\bx) & f_{1}(\bx)- \sum_{i=1}^d x_i\partial_i f_1(\bx)\\
    0 & 0 & \cdots & 0 & 1 & 0 & \cdots & x_{d}\\
     \vdots & \vdots & \vdots & \vdots & \vdots & \vdots & \vdots & \vdots \\
     0 & 0 & \cdots & 0 & 0 &  \cdots & 1 &  x_{1}\\
     0 & 0 & \cdots & 0 & 0 &  \cdots & 0 &  1    
\end{bmatrix}.\end{equation}

Let us also introduce the following set closely related to the counting function $N^\star_{\ta}(B;\varepsilon, t)$. For every $\varepsilon>0, B\subset \bU\subset \R^d$ and $t\in \N$, we define 
\begin{equation}\label{eq:N with one var}
  \mathcal{N}_{\ta}(t,\varepsilon, B):=\left\{(q,\bp_d,\bp_m)\in \N\times \Z^d\times\Z^m~\left|~\begin{aligned} &\frac{\bp_d+\ta_d}{q}\in B,\; e^{t-1}\leq q\leq e^t\\
  &\left\Vert q\fff\left(\frac{\bp_d+\ta_d}{q}\right)-\ta_m-\bp_m\right\Vert
  \leq \varepsilon\end{aligned}
  \right.
  \right\}\,,
\end{equation}
where $f$ is given as in \eqref{Monge} and $\ta=(\ta_d,\ta_m)$ with $\ta_d\in\R^d$ and $\ta_m\in\R^m$.
When $t>0$ is large enough and $B\subset \bU,$ one can see using Taylor's expansion that
\begin{equation}\label{connection} \#\mathcal{N}_{\ta}(t,\varepsilon, B)\ll N^\star_{\ta}(B;\varepsilon, t).\end{equation} 


\section{Lower bounds}

First we recall some definitions from \cite{BVVZ2018} adapted to our needs and notation. 
Given $\varepsilon>0$, $\rho>0$, $t\in\N$ and $B\subset \bU$, define
\begin{equation}
    \Delta_{\ta}(t,\varepsilon, B, \rho):=\bigcup_{(q,\bp_d,\bp_m)\in \mathcal{N}_{\ta}(t,\varepsilon, B)} B\left(\frac{\bp_d+\ta_d}{q},\rho\right),
\end{equation}
where $\mathcal{N}_{\ta}(t,\varepsilon, B)$ is given by \eqref{eq:N with one var}. Further, given also $\cc>0$, define
$$
a_{\varepsilon, t, \cc}:=\cc^{-1}\diag\{\underbrace{\varepsilon,\dots,\varepsilon}_m,\underbrace{(\varepsilon^m e^t)^{-1/d},\dots,(\varepsilon^m e^t)^{-1/d}}_d,\cc^{n+1}e^t\}
$$
and the set
\begin{equation}
    \mathfrak{G}(\cc, t, \varepsilon):=\{\bx\in \bU~:~ \delta_1(a_{\varepsilon, t, \cc}^{-1} u_1(\bx)\Z^{n+1})\ge \cc\}.
\end{equation}

The matrix $u_1(x)$ coincides with the matrix $G$ defined in \cite{BVVZ2018} up to some permutations of rows and columns and sign change in some columns. Hence, the set $\mathfrak{G}(\cc, t, \varepsilon)$ given above coincides with the set $\mathcal{G}(c,Q,\psi)$ defined in \cite{BVVZ2018} upon setting $c=\cc^{n+1}$, $Q=e^t$ and $\psi=\varepsilon$. In particular, we can now restate Corollary 2.2 from \cite{BVVZ2018} as follows.

\begin{lemma}\label{lemma:BVVZ2.1}
Let $\f,\bU$ be as before. Suppose that $\cc\in(0,1]$, $M$, $\tilde t$, $\tilde\varepsilon>0$ are given such that \begin{equation}\label{eqn_lower_bound 2 on eps}
        \tilde\varepsilon\geq K_0 e^{-\frac{\tilde t(d+2)}{2m+d}},\qquad \text{ with } K_0\gg_{n,d}\frac{1}{\cc^{n+1}}\left( 1+ \frac{M}{\cc^{n+1} }\right). 
    \end{equation}
Suppose that \eqref{condi2} holds. Let $t$, $\varepsilon$ and $\rho$ be defined from the following equations: \begin{equation}\label{eqn: t and eps}
    e^t=\frac{e^{\tilde t}}{4(n+1)},\qquad  \varepsilon=\frac{\tilde\varepsilon}{(1+ \frac{Md^2}{\cc^{n+1}})\frac{n+1}{\cc^{n+1}}},
\end{equation}
\begin{equation}
\rho=\frac{1}{2\cc^{n+1}}\left(\varepsilon^m e^{t(d+1)}\right)^{-1/d} = C_0\left(\tilde\varepsilon^m e^{\tilde t(d+1)}\right)^{-1/d},
\end{equation} 
where $C_0>0$ depends on $n$, $m$, $d$, $M$ and $\cc$ only.
Then for any ball $B\subset \bU,$ we have that 
\begin{equation}
     \mathfrak{G}(\cc, t, \varepsilon)\cap B^{\rho}\subset \Delta_{\ta}(\tilde t, \tilde\varepsilon, B, \rho),
\end{equation}
    where $B^\rho$ contains all $x\in B$ such that $B(x,\rho)\subset B.$
\end{lemma}

By \eqref{eqn: Mincon}, for any matrix $g\in\mathrm{SL}_{n+1}(\R)$ if $\delta_1(g\Z^{n+1})\le \cc$ then
$$
\delta_{n+1}(g\Z^{n+1})\gg_n \cc^{-1/n}\,.
$$
Hence
\begin{align}
\nonumber\bU\setminus \mathfrak{G}(\cc, t, \varepsilon)&=
\{\bx\in \bU~:~ \delta_1(a_{\varepsilon, t, \cc}^{-1} u_1(\bx)\Z^{n+1})<\cc\}\\[2ex]
\nonumber&\subset\{\bx\in \bU~:~ \delta_{n+1}(a_{\varepsilon, t, \cc}^{-1} u_1(\bx)\Z^{n+1})\gg_n  \cc^{-1/n}\}\\[2ex]
&\stackrel{\text{Lemma } \ref{dual lemma}}{\subset}\{\bx\in \bU~:~ \delta_{1}((a_{\varepsilon, t, \cc}^{-1})^* u_1^*(\bx)\Z^{n+1})\ll_n  \cc^{1/n}\}\,.\label{eq5.8}
\end{align}
Note that
$$
(a_{\varepsilon, t, \cc}^{-1})^\star=\sigma^{-1}a_{\varepsilon, t, \cc}\sigma=
\cc^{-1}\diag\{\cc^{n+1}e^t,\underbrace{(\varepsilon^m e^t)^{-1/d},\dots,(\varepsilon^m e^t)^{-1/d}}_d,\underbrace{\varepsilon,\dots,\varepsilon}_m\},
$$
and observe that $u_1(\bx)^\star$ is given by 
    \begin{equation}\label{eqn:dual u_1(x)}
        u_1^\star(\bx)=\begin{bmatrix}
            1 & -\bx & -\fff(\bx)\\
            0 & \mathrm{I}_d & J(\bx)\\
            0 & 0   & \mathrm{I}_m
        \end{bmatrix}\,.
    \end{equation}
Therefore, there is a constant $c_n>0$ depending on $n$ only, and $c_{n,\f}>0$ depending on $n$ and $\f$ such that \eqref{eq5.8} is contained in $\mathfrak{S}_{\f}(\delta,K,\bT)$ with
$$
\delta=\cc^{1/n}c_n\cc^{-n}e^{-t},\quad K=\cc^{1/n}c_n\cc(\varepsilon^m e^t)^{1/d},\quad T_1=\dots=T_n=\cc^{1/n}c_{n,\f} \cc\varepsilon^{-1}\,.
$$ 
Now we state the key proposition in this section, which follows from Theorem \ref{BKM_new} following the discussion above.

\begin{proposition}\label{key_propo}
Let $\f:\bU\subset \R^d \to\R^n$ be a map of the form \eqref{Monge}, where $\bU$ is an open set in $\R^d.$ Suppose that $\f$ is nondegenerate at $\bx_0\in\bU$. Let $0<\kappa<1$. Then there exists a sufficiently small neighborhood $V\subset\bU$ of $\bx_0$, and constants $0<\cc<1$ and $K_0>0$ such that for any ball $B\subset V$, there exists $t_0=t(B, n,\f,\kappa)$ such that for all $t\geq t_0$, and $K_0 e^{-\eta t}<\varepsilon<1$,  where $\eta$ is as in \eqref{def:eta}, we have that 
    \begin{equation}
        \Leb_d(B\setminus  \mathfrak{G}(\cc, t, \varepsilon))\leq \kappa \Leb_d(B)\,.
    \end{equation}
\end{proposition}

\begin{theorem}\label{main_div_1}
Let $\f:\bU\subset \R^d \to\R^n$ be a map of the form \eqref{Monge}, where $\bU$ is an open set in $\R^d$, $\ta\in\R^n$. Suppose $\f$ is nondegenerate at $\bx_0\in\bU.$  Then there exists a small open ball $V$ centered at $\bx_0$, and constants $C_0, K_0>0$ such that for every ball $B\subset V$, there exists $t_0=t_0(B,n, \f)$ such that for every $t\geq t_0$ and $\varepsilon\in(0,1)$ satisfying 
\begin{equation}\label{eqn: cond on eps}
\varepsilon> K_0 e^{-\eta t} \quad \text{with }\eta \text{ as in } \eqref{def:eta}
\end{equation}
we have that
\begin{equation}
    \Leb_d(\Delta_{\ta}(t, \varepsilon, B, \rho))\geq  \tfrac{1}{2}\Leb_d(B),
\end{equation}
where \begin{equation}\label{eqn: define of rho}
    \rho=\left(\frac{C_{0}}{\varepsilon^{m}e^{(d+1)t}}\right)^{-\frac{1}{d}}.
\end{equation}
\end{theorem}
\begin{proof}
    Suppose from Proposition \ref{key_propo} we get $V\subset \bU$ centered at $\bx_0$, and $K^1_0>0,\delta>0$ for $\kappa= \frac{1}{3}$. Let us take $K_0>K^1_0$ and $K_0$ satisfies \eqref{eqn_lower_bound 2 on eps}. Using Proposition \ref{key_propo}, for every ball $B\subset V$, there exists $t_0>0$ such that for all $t\geq t_0$, 
\begin{equation}\label{eqn:upper bound complement}
 \Leb_d(B\setminus \mathfrak{G}(\delta, t, \varepsilon) )\leq \tfrac{1}{3} \Leb_d(B).
\end{equation}
Now applying Lemma \ref{lemma:BVVZ2.1}, for $\tilde\varepsilon$ and $\tilde t$ as in \eqref{eqn: t and eps}, since $e^{-\eta t}\geq e^{-\frac{t(d+2)}{2m+d}}$ we get that
\begin{equation}\label{lower_bound}
\Leb_d(\Delta(\tilde t, \tilde\varepsilon, B, \rho)) \geq \Leb_d(\mathfrak{G}(\delta, t, \varepsilon)\cap B^{\rho})\geq \Leb_d(B\cap \mathfrak{G}(\delta, t, \varepsilon))-\Leb_d(B\setminus B^{\rho}).
\end{equation}
By \eqref{eqn: cond on eps} and \eqref{eqn: define of rho}, we have that $\rho\to0$ as $t\to\infty$. Hence $\Leb_d(B\setminus B^{\rho})\to0$ as $t\to\infty$ and hence for all sufficiently large $t$ we have that $\Leb_d(B\setminus B^{\rho})\le\tfrac16\Leb_d(B)$. Combining this with \eqref{eqn:upper bound complement} and \eqref{lower_bound}, we get that
   $$
\Leb_d(\Delta(\tilde t, \tilde\varepsilon, B, \rho)) \geq \tfrac{1}{2}\Leb_d(B)\,.
$$
Since $\varepsilon\asymp \tilde \varepsilon$ and $e^t\asymp e^{\tilde t}$, the required result follows.
\end{proof}


As a corollary of the above theorem, we get the following:
\begin{corollary}\label{coro: lower}
    Let $\f,\ta,\bU, B, \bx_0, \varepsilon, t, C_0$ be the same as in Theorem~\ref{main_div_1}. Then 
    \begin{equation}
        \# \mathcal{N}_{\ta}(t,\varepsilon, B) \geq  \frac{\varepsilon^m e^{(d+1)t}}{2 V_dC_0} \Leb_d(B),
    \end{equation}
    where $V_d:=\Leb_d(B(0,1)).$
\end{corollary}
\subsection{Proof of Theorem~\ref{mainthm: lower bound}} Combining Corollary \ref{coro: lower} and \eqref{connection} the proof follows. 
\subsection{ Proof of Theorem~\ref{main: Haus_div}}
 
The proof of Theorem~\ref{main: Haus_div} goes exactly the same as \cite[pp. 196-199]{Ber12} and using Theorem~\ref{main_div_1}. We should remark that the range of $s$ is determined by the range of $\varepsilon$ in Theorem~\ref{main_div_1}.

\subsection{ Proof of divergence part in Theorem~\ref{thm: main}}\label{proof: div}
    Let us take  $$\psi_1(q):= \max\left\{ \psi(q), q^{-(
    \eta-\zeta)}\right\},$$ where $\eta-\frac{1}{n}>\zeta>0$ and $\eta$ is as in \eqref{def:eta}. Then  $\psi_1(q)^{1/\eta} q\to\infty$ and  $\sum\psi(q)^n=\infty$ implies that $\sum \psi_1(q)^n=\infty.$ Since $\eta-\zeta> \frac{1}{n}$, by \cite[Corollary A]{BVsim}, $\Leb_d\left( \f^{-1}(\mathcal{S}_n^{\ta}(q\to q^{-\eta}))\right)=0.$ Also note that $\f^{-1}(\mathcal{S}_n^{\ta}(\psi_1))= \f^{-1}(\mathcal{S}_n^{\ta}(\psi))\cap \f^{-1}(\mathcal{S}_n^{\ta}(q \to q^{-\eta})).$ Now applying Theorem~\ref{main: Haus_div} for $s=d$, and $\psi_1$ yields the divergence part of Theorem~\ref{thm: main}.

\bigskip

\section{Upper bounds}

\subsection{Preliminaries}\label{section: definition}\label{sec:new major}
For $\bx_0\in\bU, 1>\varepsilon>0,$ $t>0$, let us denote $\Delta_t(\bx_0):= B\left[\bx_0, \left(\frac{\varepsilon}{e^t}\right)^{\frac{1}{2}}\right]\subset\R^d$. 

Also, let
$$g_{\varepsilon,t}:=\phi\diag(\underbrace{\varepsilon^{-1},\dots,\varepsilon^{-1}}_{n}, e^{-t}),$$
\begin{equation}\label{relations} \phi:= \left(\varepsilon^n e^t\right)^{1/n+1},  h:=\frac{dt}{2(n+1)}, \end{equation}
and 
\begin{equation}\label{def: b}
b_t:=\diag\{\underbrace{e^{\frac{dt}{2(n+1)}},\ldots,e^{\frac{dt}{2(n+1)}}}_m, \underbrace{e^{\frac{-(m+1)t}{2(n+1)}},\ldots, e^{\frac{-(m+1)t}{2(n+1)}}}_d, e^{\frac{dt}{2(n+1)}} \}\in \mathrm{SL}_{n+1}(\R).
\end{equation}
Let us define $d_\varepsilon=\diag(\underbrace{1,\dots,1}_m, \underbrace{\varepsilon^{1/2},\dots,\varepsilon^{1/2}}_d, 1)$ for $\varepsilon>0$. 
The following equation follows from matrix multiplication.

\begin{equation}\label{lemma: matrix multiplication} g_{\varepsilon,t}z({\bx})g_{\varepsilon,t}^{-1}= z({\bx}). 
\end{equation}

Now we make  
modified definitions of generic and special parts from \cite{BY}. Let $d_{\varepsilon}, g_{\varepsilon,t}, b_t, u_1(\bx)$ be as above. We define `raw nongeneric' set to be 
$$
\mathfrak{M}^{new}_0(\varepsilon,t):=\{\bx\in \bU~|~\delta_{n+1}(d_{\varepsilon}b_t g_{\varepsilon,t}u_1(\bx)\Z^{n+1})>\phi e^{h}\}.
$$
 
Note the above definition uses a new matrix $d_{\varepsilon}$ that was not used in \cite{BY}; see \cite[Equation 4.18]{BY}.

Then we define the special part to be a thickening of the above raw set, namely, 
\begin{equation}\label{def: minor}\mathfrak{M}^{new}(\varepsilon,t):= \bigcup_{\bx\in \mathfrak{M}^{new}_0(\varepsilon,t)} B(\bx, \varepsilon^{1/2} e^{-t/2})\cap\bU.\end{equation} 
Note that the above thickening radius is also different than the same in \cite[Equation 4.19]{BY}.
Lastly, let us recall that the generic part is the complement of the special part, i.e., 
$$\bU\setminus \mathfrak{M}^{new}(\varepsilon,t).$$

 Similar to \cite[Lemma 1.4]{BY} we have,
\begin{lemma}\label{l1}
If $\f(\bx)\in \mathcal{S}^{\ta}_n(\psi)$ then there are infinitely many $t\in\N$ such that \begin{equation}
    \left\Vert \f(\bx)-\frac{\bp+\ta}{q}\right\Vert<\frac{\psi(e^{t-1})}{e^{t-1}}
\end{equation} for some $(q, \bp)\in \Z^{n+1}$ with $e^{t-1}\leq q<e^t$.
\end{lemma}

The proof of the above lemma uses the fact that $\psi$ is monotonically decreasing. Our next lemma follows from simple matrix multiplication. 
\begin{lemma}\label{lemma: dani3}
Let $\f, \bU$ be as before. Suppose for some $\bx\in \bU,$ and $\varepsilon>0$, and $t\in\N,$ there exists $(q, \bp)\in\Z\times \Z^{n}$ such that,
\begin{equation}\label{eqn: number th1}
    \begin{aligned}
&\left\vert q \fff_j(\bx)-\sum_{i=1}^d\partial_i\fff_j(\bx)(qx_{i}-p_i)-p_{d+j}\right\vert\ll_{\f, n}\varepsilon,~ 1\leq j\leq m,\\
& \vert q x_{i}-p_i\vert\ll_{\f, n}\varepsilon^{\frac{1}{2}} e^{t/2},~ 1\leq i\leq d,\\
& \vert q\vert \ll_{\f, n} e^t.
    \end{aligned}
\end{equation}
Then \begin{equation}
     g_{\varepsilon,t}u_1(\bx)(-\bp\sigma_n, q)
 \in [c\phi]^m\times [c \phi(\eptm)]^d\times [c\phi].
\end{equation} where  $\phi,\varepsilon,t$ are related as \eqref{relations} and $c$ depends on $\f$ and $n$. 
\end{lemma}
The proof is obvious from the definition of $u_1(\bx)$ in \eqref{def_u_1}. We assume $\f,\bU, n, m, d, \ta$ be as in Theorem~\ref{thm: main}. By \eqref{Monge} and \eqref{condi2} the following lemma follows easily.

\begin{lemma}\label{lemma: Dani}
Suppose for some $\bx\in\bU,$ $t>0,\varepsilon>0,$ \begin{equation}\label{eqn: inhomo}\Vert q\f(\bx)-\bp+\ta\Vert<\varepsilon,\end{equation} for some $(q,\bp)\in\Z^{n+1}$, with $0<q\leq e^t$. 
Then,
\begin{equation}\label{eqn: Dani1}
    \begin{aligned}
&\left\vert q \fff_j(\bx)-\sum_{i=1}^d\partial_i\fff_j(\bx)(q x_{i}-p_i+\theta_i)-p_{d+j} +\theta_{d+j}\right\vert\ll_{\f,n}\varepsilon,~1\leq j\leq m,\\
& \vert q x_{i}-p_i+\theta_i\vert\leq\varepsilon,~1\leq i\leq d,\\
& \vert q\vert \leq e^t.
    \end{aligned}
\end{equation}

\end{lemma}

\begin{lemma}\label{lemma:useful}
Let $1\geq \varepsilon>0, t>0$. Suppose $\bx\in \Delta_t(\bx_0):= B\left[\bx_0, \left(\frac{\varepsilon}{e^t}\right)^{\frac{1}{2}}\right]\subset\R^d$, and $(\bp,q)\in\Z^{n+1}$, with $1\leq q\leq e^t$ such that $\Vert \f(\bx)-\frac{\bp+\ta}{q}\Vert<\frac{\varepsilon}{e^t}.$ Then 
\begin{equation}
    \begin{aligned}
&\left\vert q \fff_j(\bx_0)-\sum_{i=1}^d\partial_i\fff_j(\bx_0)(q x_{0,i}-p_i+\theta_i)-p_{d+j} +\theta_{d+j}\right\vert\ll_{\f, n}\varepsilon, ~1\leq j\leq m,\\
& \vert q x_{0,i}-p_i+\theta_i\vert\leq 2\varepsilon^{1/2} e^{t/2},~ 1\leq i\leq d,\\
& \vert q\vert \leq e^t.
    \end{aligned}
\end{equation}

\end{lemma}
\begin{proof}
    Using the hypothesis of this lemma, by Lemma \ref{lemma: Dani}, \eqref{eqn: Dani1} holds.   One can write $\bx=\bx_0+\left(\varepsilon e^{-t}\right)^{1/2}\bx'$, where $\Vert \bx'\Vert\leq 1.$ For $1\leq i\leq d,$
    $$\begin{aligned}&\vert q x_{0,i}-p_i+\theta_i\vert\\ &<\vert q x_{i}-p_i+\theta_i\vert+ \vert q\vert \vert x_i-x_{0,i}\vert\\
    &\leq \varepsilon+ e^t \ept\leq 2\varepsilon^{1/2} e^{t/2}.\end{aligned}$$

    Also, by \eqref{eqn: Dani1}, for $1\leq j\leq m,$
$$\begin{aligned}
&\left\vert q \fff_j(\bx_0)-\sum_{i=1}^d\partial_i\fff_j(\bx_0)(q x_{0,i}-p_i+\theta_i)-p_{d+j} +\theta_{d+j}\right\vert\\
& \ll_{\f,n}\varepsilon+\sum_{i=1}^d\vert \partial_i\fff_j(\bx_0)-\partial_i\fff_j(\bx) \vert \vert qx_i-p_i+\theta_i\vert +\vert q\vert \left(\varepsilon e^{-t} \right)\\
& \ll_{\f,n} \varepsilon+\ept \varepsilon+ \varepsilon\\
& \ll_{\f,n}\varepsilon.
\end{aligned}
$$

To conclude some of the inequalities above, we use Taylor's expansion, $\fff_j(\bx_0)=\fff_j(\bx)+ \sum_{i=1}^d \partial_i\fff_j(\bx)(x_{0,i}-x_i)+\sum_{\beta\in\N^2,\vert \beta\vert=2}\hat\fff_{j,\beta}(\bx)(\bx_{0}-\bx)_\beta,$ where for $\by\in \R^d,$ $\by_\beta=y_k y_l$ if $\beta=(k,l)\in \N^2, \vert \beta\vert:=k+l=2.$ Similarly, we also use that $
\partial_{i}\fff_j(\bx_0)=\partial_i\fff_j(\bx)+ \sum_{k=1}^d \hat\fff_{i,j,k}(\bx)(x_{0,k}-x_k)$ and for every $\bx\in\bU$, \begin{equation}\label{taylorconsequence}
\vert \hat\fff_{i,j,k}(\bx)\vert \leq  \max_{\by\in\bU} \vert \partial_{k,i}^2\fff_j(\by)\vert \text{ and }
\vert \hat\fff_{i,\beta}(\bx)\vert \leq \max_{\by\in\bU}\vert \partial_{\beta}\fff_i(\by)\vert.\end{equation}
\end{proof}

\subsection{Generic part}
Let us recall the definition of $\mathfrak{M}^{new}(\varepsilon,t)$ from \S \ref{sec:new major}, and we continue with the same notations as in the previous sections.
\begin{lemma}\label{lemma: inhomo count}
Let $B$ be a ball in $\bU.$ For all large enough $t>0,$ $0<\varepsilon<1$ for every $\bx_0\in B\cap (\bU\setminus \mathfrak{M}^{new}(\varepsilon,t)),$ we have 
$$N_{\ta}(\Delta_t(\bx_0)\cap B;\varepsilon, t)\ll_{n,\f} \varepsilon^n e^t (\varepsilon e^{-t})^{\frac{-d}{2}}.$$

\end{lemma}
\begin{proof} 
Without loss of generality we assume $N_{\ta}(\Delta_t(\bx_0)\cap B;\varepsilon, t)>1.$
Then take two distinct $(\bp,q)$ and $(\bp',q')$ in  $\Z^{n}\times \N$ such that $1\leq q,q'\leq e^{t}$
and  for some $\bx,\bx'\in \Delta_t(\bx_0)\cap B$, 
$$ \left\Vert \f(x)-\frac{\bp+\ta}{q}\right\Vert\leq \frac{\varepsilon}{e^t}, \text{ and } \left\Vert \f(x')-\frac{\bp'+\ta}{q'}\right\Vert\leq \frac{\varepsilon}{e^t}.$$
By Lemma \ref{lemma:useful}, we have 
\begin{equation}\label{eqn: inhomo q}
    \begin{aligned}
&\left\vert q \fff_j(\bx_0)-\sum_{i=1}^d\partial_i\fff_j(\bx_0)(q x_{0,i}-p_i+\theta_i)-p_{d+j}+\theta_{d+j}\right\vert\ll_{\f,n}\varepsilon, ~1\leq j\leq m,\\
& \vert q x_{0,i}-p_i+\theta_i\vert\leq 2\varepsilon^{1/2} e^{t/2},~ 1\leq i\leq d,\\
& \vert q\vert \leq e^t,
    \end{aligned}
\end{equation}
and 
\begin{equation}\label{eqn: inhomo q'}
    \begin{aligned}
&\left\vert q' \fff_j(\bx_0)-\sum_{i=1}^d\partial_i\fff_j(\bx_0)(q' x_{0,i}-p'_i+\theta_i)-p'_{d+j}+\theta_{d+j}\right\vert\ll_{\f,n}\varepsilon,~ 1\leq j\leq m,\\
& \vert q' x_{0,i}-p'_i+\theta_i\vert\leq 2\varepsilon^{1/2} e^{t/2},~1\leq i\leq d,\\
& \vert q'\vert \leq e^t.
    \end{aligned}
\end{equation}
Now by subtracting the previous two equations, we get 
\begin{equation}\label{eqn: inhomo q''}
    \begin{aligned}
&\left\vert q'' \fff_j(\bx_0)-\sum_{i=1}^d\partial_i\fff_j(\bx_0)(q'' x_{0,i}-p^{''}_i)-p''_{d+j}\right\vert\ll_{\f, n}\varepsilon,~ 1\leq j\leq m,\\
& \vert q'' x_{0,i}-p''_i\vert\ll 4\varepsilon^{1/2} e^{t/2},~1\leq i\leq d,\\
& \vert q''\vert \leq 2 e^t,
    \end{aligned}
\end{equation}
where $q''= q-q'$ and $\bp''=\bp-\bp'.$
 
 Now by Lemma \ref{lemma: dani3} we have, 
 \begin{equation}\label{eqn: homofinish1}
 g_{\varepsilon,t}u_1(\bx_0)(-\bp''\sigma_n, q'')
 \in [c\phi]^m\times [c \phi(\eptm)]^d\times [c\phi],
 \end{equation} where $c$ depends on $\f$ and $n.$
 From \eqref{eqn: homofinish1}, and by the definition of $b_t$ as in \S \ref{section: definition}, we have 
 $$b_t g_{\varepsilon,t} u_1(\bx_0)(-\bp''\sigma_n,q'')\in \underbrace{[c_1\phi e^h]^m\times [c_1\phi e^h \varepsilon^{-1/2}]^d\times [c_1 \phi e^h]}_{\Omega} \cap~ b_t g_{\varepsilon,t} u_1(\bx_0)\Z^{n+1} ,$$ where $c_1>0$ depends on $n,\f.$ 
 
 Now applying $d_\varepsilon$ we get 
$$
d_{\varepsilon}b_t g_{\varepsilon, t}u_1(\bx_0)(-\bp'' \sigma_n, q'')\in \underbrace{[c_1\phi e^h]^m \times [c_1 \phi  e^h]^d\times [c_1 \phi e^h]}_{\Omega^\star} \cap~ d_{\varepsilon} b_t g_{\varepsilon,t} u_1(\bx_0)\Z^{n+1}.
$$ Since $\bx_0\in \bU\setminus \mathfrak{M}^{new}(\varepsilon, t)$, we have 
$$
\delta_{n+1}(d_{\varepsilon}b_t g_{\varepsilon, t}u_1(\bx_0)\Z^{n+1})\leq \phi e^h.
$$
Hence $c_2\Omega^\star$ for some large $c_2>0$ will contain the fundamental domain of the lattice $d_\varepsilon b_t g_{\varepsilon,t} u_1(\bx_0)\Z^{n+1}$. 
Therefore, 
$$\# c_2 \Omega^\star \cap d_{\varepsilon} b_t g_{\varepsilon, t}u_1(\bx_0)\Z^{n+1}\ll_{n,\f} \frac{\Leb_{n+1}( \Omega^\star)}{\det d_{\varepsilon}}= \frac{\phi^{n+1}  e^{h(n+1)}}{\varepsilon^{d/2}}= \phi^{n+1} \varepsilon^{-d/2} e^{h(n+1)}= \varepsilon^n e^t (\varepsilon e^{-t})^{-d/2}.$$
Thus the number of such $(\bp'',q'')$ is $\ll \varepsilon^n e^t (\varepsilon e^{-t})^{-d/2}$, and note that we should get at least $ N_{\ta}(\Delta_t(\bx_0)\cap B;\varepsilon,t)-1$ many 
(distinct) of them, which therefore concludes the lemma.
\end{proof}


The proof of the following lemma is exactly the same as in \cite[Lemma 5.4]{BY}. 

\begin{lemma}\label{lemma: covering count}
For sufficiently large $t>0,$
$$N_{\ta}(B\setminus \mathfrak{M}^{new}(\varepsilon,t); \varepsilon,t)\leq 2 {(\varepsilon e^{-t})}^{-d/2}\Leb_d(B) \max_{\bx_0\in B\cap (\bU\setminus \mathfrak{M}^{new}(\varepsilon,t))}N_{\ta}(\Delta_t(\bx_0);\varepsilon, t).
$$ 

\end{lemma}

As a consequence of Lemma \ref{lemma: inhomo count}, and Lemma \ref{lemma: covering count}, we have the following proposition.
\begin{proposition}\label{Propo: semi}
    Let $\f$, $\bU$ be as before. Then for any $0<\varepsilon<1,$
    for any ball $B\subset \bU$, and all sufficiently large $t>0,$
    we have 
    \begin{equation}
        N_{\ta}(B\setminus \mathfrak{M}^{new}(\varepsilon,t); \varepsilon,t)\ll_{n,\f} \varepsilon^m e^{(d+1)t}\Leb_d(B).
    \end{equation}
\end{proposition}

\subsection{Nongeneric part} First, for some matrices introduced in previous sections, we compute their dual matrices. For $\varepsilon>0, t>0,$ 
\begin{equation}
g_{\varepsilon,t}^\star= \phi^{-1}\diag(e^t,\underbrace{\varepsilon,\dots,\varepsilon}_{n}),
    \end{equation}
    \begin{equation}\label{eqn:dual u_1(x)again}
        u_1^\star(\bx)=\begin{bmatrix}
            1 & -\bx & -\fff(\bx)\\
            0 & \mathrm{I}_d & J(\bx)\\
            0 & 0   & \mathrm{I}_m
        \end{bmatrix},
    \end{equation}
$$d^\star_\varepsilon=\diag(1,  \underbrace{\varepsilon^{-1/2},\dots,\varepsilon^{-1/2}}_d, \underbrace{1,\dots,1}_m),$$ and $$b^\star_t=\diag(e^{-h},  \underbrace{e^{\frac{(m+1)t}{2(n+1)}},\dots, e^{\frac{(m+1)t}{2(n+1)}}}_d, \underbrace{e^{-h},\dots,e^{-h}}_m).$$

The next lemma shows that the new special part is inside the set defined in \eqref{eqn5.1}.

\begin{lemma}\label{inclusion}
For $\varepsilon>0, t>0$ and $\f$ as before,
   $$ \mathfrak{M}^{new}(\varepsilon,t)\subset \mathfrak{S}_{\f}( ce^{-t}, c\varepsilon^{-1/2} e^{-t/2}, c\varepsilon^{-1}),$$ where $c$ depends on $n,\f,$ and the above sets are defined in \eqref{eqn5.1} and \eqref{def: minor}.
\end{lemma}
\begin{proof}
    First note that, if $\bx_0\in \mathfrak{M}_0^{new}(\varepsilon,t)$ then 
    $$\delta_{n+1}(d_\varepsilon b_t g_{\varepsilon, t} u_1(\bx_0)\Z^{n+1})>\phi e^{h}, \text{ where } h=\frac{dt}{2(n+1)}.$$
Then by Lemma \ref{dual lemma}, and the fact that $(g_1 g_2)^\star=g_1^\star g_2^\star,$ $$\delta_1(d^\star_\varepsilon b^\star_t g^\star_{\varepsilon, t} u^\star_1(\bx_0)\Z^{n+1})\ll_{n}\phi^{-1} e^{-h}.$$
    This implies there exists $(a_0, \ba)\in \Z^{n+1}\setminus \mathbf{0}$ such that 
    
    \begin{equation}\label{raw_BKM}
    \begin{aligned}
    &\vert a_0+\f(\bx_0)\ba^\top\vert\ll_{n} e^{-t}\\
    & \Vert \nabla\f(\bx_0)\ba^\top\Vert\ll_n \varepsilon^{-1/2} e^{-t/2}\\
    & 0<\Vert \ba_m\Vert \ll_n \varepsilon^{-1},
    \end{aligned}
    \end{equation} where $\ba_m=(a_{d+1},\dots, a_n).$
Note that using the second and third inequalities in \eqref{raw_BKM} and by 
\eqref{Monge} we get $\Vert \ba\Vert \ll_{n, \f} \varepsilon^{-1}.$
    
    Now for any $\bx\in \mathfrak{M}^{new}(\varepsilon,t),$ there exists $\bx_0\in \mathfrak{M}_0^{new}(\varepsilon,t)$ such that $\Vert \bx-\bx_0\Vert\leq \varepsilon^{1/2} e^{-t/2}$, we have by Taylor's expansion, 
    $$\vert a_0+\f(\bx)\ba^\top\vert \ll_{n,\f} e^{-t}+ \varepsilon^{-1/2} e^{-t/2} \varepsilon^{1/2} e^{-t/2}+ \varepsilon^{-1} (\varepsilon^{1/2} e^{-t/2})^2\ll_{n,\f} e^{-t}.$$
    Similarly, using Taylor's expansion we get, 
    $$
    \Vert \nabla\f(\bx)\ba^\top\Vert\ll_{n,\f} \varepsilon^{-1/2} e^{-t/2}+ \varepsilon^{-1} \varepsilon^{1/2} e^{-t/2}\ll_{n,\f}
    \varepsilon^{-1/2} e^{-t/2}.
    $$
    Thus the lemma follows.

\end{proof}
Using Lemma \ref{inclusion} and Theorem~\ref{thm:KM} we have the following proposition.

\begin{proposition}\label{Bound on minor}
    Let $\bU$ be an open set in $\R^d$, $\f:\bU\to\R^n$ be a map satisfying \eqref{Monge}. Let $\f$ be $l$-nondegenerate at $\bx_0\in \bU.$ Then there exists a ball $B_0$ centered at $\bx_0$, constants $K_0, t_0$ depending on $n,\f, B_0$ such that  for $t\geq t_0,$
    $$
    \Leb_{d}(\mathfrak{M}^{new}(\varepsilon, t)\cap B_0)\ll K_0 (e^{-t} \varepsilon^{-1/2} e^{-t/2} \varepsilon^{-(n-1)})^{\alpha}\Leb_d(B_0)= (e^{3t/2} \varepsilon^{n-1/2})^{-\alpha}\Leb_d(B_0),
    $$ where $\alpha= \frac{1}{d(2l-1)(n+1)}.$
\end{proposition}

\subsection{Completing the proof convergence}
Combining Proposition \ref{Propo: semi} and Proposition \ref{Bound on minor}, we have the key proposition as follows.

\begin{proposition}\label{propo: main }
Let $\bU\subset \R^d$ be an open set and $\f:\bU\to\R^n$ be a map satisfying \ref{Monge}. Then for any $0<\varepsilon<1$ and every $t>0$ there exists $\mathfrak{M}^{new}(\varepsilon, t)\subset \bU,$ which can be written as a union of balls of radius $\varepsilon^{1/2} e^{-t/2}$ of intersection multiplicity $N_d.$ If $\f$ is nondegenerate at $\bx_0$ then there is a ball $B_0\subset \bU$ centered at $\bx_0$ and constants $K_0, t_0>0$, which depend on $n, \f$ and $B_0$ only, such that for $t\geq t_0$, 
\begin{equation}\label{bound_minor_crucial}
    \Leb_d( \mathfrak{M}^{new}(\varepsilon, t)\cap B_0)\leq K_0 \left(\varepsilon^{n-1/2} e^{3t/2}\right)^{-\alpha},  \alpha= \frac{1}{d(2l-1)(n+1)},
\end{equation}
and for every ball $B\subset \bU$ and sufficiently large $t$, we have 
\begin{equation}
N_{\ta}(B\setminus \mathfrak{M}^{new}(\varepsilon, t); \varepsilon, t) \ll_{n,\f} \varepsilon^{m} e^{(d+1)t} \Leb_d(B).
\end{equation}

\end{proposition}

Now, we can complete the proof of Theorem~\ref{mainthm: upperbound}, Theorem~\ref{thm: main} and \ref{Thm: Hausdorff measure} using the above proposition.

\subsection{Proof of Theorem~\ref{mainthm: upperbound}}
    Let $\alpha, \f$, $B_0$, and $ \bx_0$ be as in Proposition \ref{propo: main }. For $1>\varepsilon>0$ and all large $t>0$,
    \begin{equation}
        N_{\ta}(B_0;\varepsilon,t)\ll_{n,\f, B_0}\varepsilon^m e^{(d+1)t}+ e^{(d+1)t} \left(\varepsilon^{n-1/2} e^{3t/2}\right)^{-\alpha}.
    \end{equation} 
To get the second term above, we cover $\mathfrak{M}^{new}(\varepsilon,t)\cap B_0$ by $\ll \frac{\Leb_d(\mathfrak{M}(\varepsilon, t)\cap B_0)}{\varepsilon^{d/2} e^{-td/2}}$ balls of radius $\varepsilon^{1/2} e^{-t/2}$.

\subsection{Proof of convergence part in Theorem~\ref{thm: main}}
Since it is enough to prove this theorem for approximating function being $\max\{\psi(q), q^{-\frac{5}{4n}}\}$, without loss of generality, we assume that for all $q\in\N$,
\begin{equation}\label{condition on si}
\psi(q)^n>q^{-\frac{5}{4}}.
\end{equation}
For any $\Delta\subset \R^d$, any $t>0$ and $0<\varepsilon<1,$ we define 
$$\mathcal{R}_{\ta}(\Delta;\varepsilon,t):=\left\{(\bp,q)\in\Z^{n+1}~|~e^{t-1}\leq q<e^t, ~\inf_{\bx\in\Delta}\left\Vert \f(\bx)-\frac{\bp+\ta}{q}\right\Vert <\frac{\varepsilon}{e^t}\right\}.$$
By definition as in \S \ref{intro_counting},
$$N_{\ta}^\star(\Delta;\varepsilon,t)= \# \mathcal{R}_{\ta}(\Delta;\varepsilon,t).$$
Now, 
$$\begin{aligned}
&\{\bx\in B_0~|~\f(\bx)\in \mathcal{S}^{\ta}_{n}(\psi)\}\subset \bigcup_{t\geq T} {\underbrace{\left(\mathfrak{M}^{new}(e\psi(e^{t-1}), t)\cap B_0 \right)}_{A_t}} \bigcup \\
& \bigcup_{t\geq T}\underbrace{\bigcup_{(\bp,q)\in \mathcal{R}_{\ta}(B_0\setminus \mathfrak{M}^{new}(e \psi(e^{t-1}), t); e\psi(e^{t-1}),t)} \left\{\bx\in B_0~|~\left\Vert \bx-\frac{\pi(\bp+\ta)}{q}\right\Vert \leq \frac{\psi(e^{t-1})}{e^{t-1}}\right\}}_{B_t^{\ta}}.
\end{aligned}$$
Here $\pi:\R^n\to\R^d$ is the projection of the first $d$ many coordinates.
By \eqref{bound_minor_crucial}, we have 
$$
\Leb_d (A_t)\ll \left(\psi(e^{t})^{n-1/2} e^{3t/2}\right)^{-\alpha}\stackrel{\eqref{condition on si}}{<}\left(e^{-(1-1/2n)5t/4} e^{3t/2}\right)^{-\alpha}= (e^{(1+5/2n)t/4})^{-\alpha},
$$
where $\alpha>0$ is as in Proposition \ref{propo: main }.
Next, observe that 
$$
\Leb_d(B_t^{\ta})\ll \psi(e^{t-1})^m e^{(d+1)t}  \left(\frac{\psi(e^{t-1})}{e^{t-1}}\right)^d\ll \psi(e^{t-1})^{n} e^{t-1}.
$$ Since $\sum \psi(q)^n<\infty,$ we have $\sum_{t\geq 1} \psi(e^t)^n e^t<\infty.$
The proof is now complete by using the convergence Borel Cantelli lemma.
\subsection{Proof of Theorem~\ref{Thm: Hausdorff measure} }\label{proof main2}
By \eqref{eqn: measure condition on degenerate}, it is enough to prove $$
 \mathcal{H}^s(\{\bx\in B_0~|~\f(\bx)\in \mathcal{S}_n^{\ta}(\psi)\})=0,$$ where $B_0$ is a very small neighborhood of $\bx_0$, and $\f$ is $l$-nondegenerate at $\bx_0.$ Let us fix $B_0$ as in Proposition \ref{propo: main }, for $\varepsilon=e\psi(e^{t-1})$, for all large $t>0$, $\mathfrak{M}^{new}(\varepsilon,t)$ can be covered as union of balls of radius $\varepsilon^{1/2} e^{-t/2}$ such that number of such balls will be $\ll \Leb_d(\mathfrak{M}(\varepsilon,t)/ (\varepsilon^{d/2} e^{-dt/2}).$ Hence by Proposition \ref{propo: main } we have that,
 $$\mathcal{H}^s(A_t)\ll \left(\varepsilon^{1/2} e^{-t/2}\right)^{s-d}\Leb_d(\mathfrak{M}^{new}(\varepsilon,t)).  $$ 
 By Proposition \ref{propo: main }, plugging $\varepsilon=e\psi(e^{t-1})$ we get that 

$$\sum_{t\geq T}\mathcal{H}^s(A_t)\ll  \sum_{t\geq T}\left(\frac{\psi(e^t
)}{e^{t}}\right)^{(s-d)/2} \left(\psi(e^t)^{n-1/2} e^{3t/2}\right)^{-\alpha}\to 0,$$ by \eqref{eqn: conv 3}, as $T\to\infty.$

Also, 
 
 $$\sum_{t\geq T}\mathcal{H}^s(B_t^{\ta})\ll \sum_{t\geq T}\varepsilon^m e^{(d+1)t} \left(\varepsilon e^{-t}\right)^{s}, \text{ where } \varepsilon=e\psi(e^{t-1})$$
and the above sum tends to zero by \eqref{eqn: conv 2} as $T\to\infty.$

\subsection{Proof of Corollary \ref{exact}}\label{proof: exact}
The upper bound follows directly from Theorem~\ref{Thm: Hausdorff measure} on applying $\psi(q)= q^{-\tau}$. The proof of this goes exactly as in the proof of \cite[Corollary 2.9]{BY} and \cite[Corollary 1.13]{SST} upon using Theorem~\ref{Thm: Hausdorff measure}. For the lower bound we check that $\tau$ as in \eqref{bad range of tau} also lies in the range $(1/n, \eta)$ as in \ref{main: Haus_div}. Let us take $\frac{(n+1)\alpha}{3\alpha+1}<\delta<\frac{n+1}{n+2}.$ Then it is easy to check that $\frac{1}{n-\delta}\in (1/n, \eta)\cap (1/n, \frac{3\alpha+1}{\alpha(2n-1)+n}).$


\bigskip

\noindent\textbf{Acknowledgments.} This work was supported by the EPSRC grant EP/Y016769/1. SD is grateful to Ralf Spatzier for their support, especially when it was needed the most. SD also thanks Subhajit Jana for numerous discussions around this project and for his constant encouragement.

\bibliographystyle{abbrv}
\bibliography{inhomosim}

\end{document}